\documentclass[ECP, preprint]{ejpecp} 

\usepackage{amsthm}
\usepackage{thmtools}
\usepackage{thm-restate}
\usepackage{color}
\declaretheorem[name=Theorem,numberwithin=section]{thm}
\usepackage{bbm}

\DeclarePairedDelimiter{\norm}{\lVert}{\rVert}
\newcommand{\indep}{\perp \!\!\! \perp}
\newcommand{\R}{{\mathbb{R}}}

\SHORTTITLE{The characteristic function of the generalised signature process}

\TITLE{A PDE approach for solving the characteristic function of the generalised signature process}

\AUTHORS{%
  Terry Lyons\footnote{Department of Mathematics, University of Oxford.
    \EMAIL{terry.lyons@maths.ox.ac.uk}}
  \and 
  Hao Ni \footnote{Department of Mathematics,
    University College London. \EMAIL{h.ni@ucl.ac.uk}}
    \and
    Jiajie Tao \footnote{Department of Mathematics,
    University College London. \EMAIL{jiajie.tao.21@ucl.ac.uk}}
    }


\AMSSUBJ{60L20; 35R45} 

\SUBMITTED{\today} 
\ACCEPTED{NA} 

\VOLUME{0}
\YEAR{2023}
\PAPERNUM{0}
\DOI{10.1214/YY-TN}

\ABSTRACT{The signature of a path, as a fundamental object in Rough path theory, serves as a generating function for non-non-commutative monomials on path space.  It transforms the path into a grouplike element in the tensor algebra space, summarising the path faithfully up to a generalised form of re-parameterisation (a negligible equivalence class in this context) \cite{Lyons2015}. Our paper concerns stochastic processes and studies the characteristic function of the path signature of the stochastic process. In contrast to the expected signature, it determines the law on the random signatures without any regularity condition. The computation of the characteristic function of the random signature offers potential applications in stochastic analysis and machine learning, where the expected signature plays an important role. In this paper, we focus on a time-homogeneous It\^o diffusion process, and adopt a PDE approach to derive the characteristic function of its signature defined at any fixed time horizon. A key ingredient of our approach is the introduction of the generalised-signature process (Definition \ref{def:sig-aug process}). This lifting enables us to establish the Feynman-Kac-type theorem for the characteristic function of the generalised-signature process by following the martingale approach. Moreover, as an application of our results, we 
present a novel derivation of the joint characteristic function of Brownian motion coupled with the L\'evy area, leveraging the structure theorem of anti-symmetric matrices.}

\begin{document}



\section{Introduction}
Rough path theory sets forth a rigorous framework to make consistent sense of controlled differential equations whether they are driven by smooth paths or driven by highly oscillatory paths. It provides a path-wise definition of the solution to stochastic differential equations in contrast to the almost sure definition from It\^o calculus. Recently, the application of rough path theory to machine learning has emerged as a mathematically principled and efficient feature representation of unparameterized streamed data. At the core of Rough path theory lies the concept of the signature, a group homomorphism from the path space onto the group-like elements in the tensor algebra space. The signature exhibits several desirable properties, including (1) faithfulness \cite{boedihardjo2015signature, Terry2010roughpath}, (2) universality \cite{levin2016learning, Terry2007roughpath}, and (3) time parameterization invariance \cite{boedihardjo2014signature, Terry2010roughpath}, making it an excellent candidate for extracting features from time series data. 

Intuitively, the signature of a path can be regarded as a generating function for the non-commutative monomials making up a path. Consequently, when the underlying path is random, the expected signature plays a role similar to that played by the moment-generating function of a random variable in $\mathbb{R}^d$. A fundamental moment problem on whether the expected signature determines the law of the underlying stochastic process has been studied in the literature. In \cite{Ilya2013char}, authors proved a sufficient condition for the expected signature to determine the law of the underlying stochastic process, i.e. if the expected signature has an infinite radius of convergence. A wide range of stochastic processes satisfy this sufficient condition. For example, the expected signature of fractional Brownian motion up to the fixed time with the Hurst parameter $H > \frac{1}{4}$. It leads to great interest in computing the expected signature. 

For example, the ``Cubature on the Wiener process" methodology \cite{Terry2004cubature} uses the expected signature of Brownian motion to build an efficient numerical solver for parabolic PDEs. The expected signature is used in the generative models for synthetic time series generation. The Sig-WGAN \cite{ni2021sigwasserstein, ni2020conditional} and Sig-MMD \cite{chevyrev2022signature} are two representative examples of GAN and MMD models respectively, which utilize the expected signature to construct a metric on the distributions induced by time series. 

Although the expected signature does characterise many measures and certainly all empirical measures, the theoretical justification of the expected signature-based generative models lies in the characteristic property of the expected signature, which may not hold in general. \cite{chevyrev2022signature} provides a simple counter-example constructed by the lognormal distribution to show that two different distributions of random paths may have the same expected signature. Furthermore, the sufficient condition for the uniqueness of the expected signature may not be satisfied. For instance, in the case of the Brownian motion upon the first exit time of a bounded domain \cite{boedihardjo2021expected, Lyons2015}, it remains unclear whether the expected signature can uniquely determine the law. 

This observation motivates us to consider the characteristic function of the signature process, which can fully characterise the distribution of the random signature without any regularity condition. It has the potential to further improve the above-mentioned methodologies based on expected signature in numerical analysis and machine learning. Moreover, it can be directly applied to derive the characteristic function of any linear function on random signatures and recover the law of the key statistics of the stochastic process, such as the L\'evy area. Similar ideas have been attempted, for instance, \cite{cuchiero2023signature} led a detailed analysis on the Laplace-Fourier transform of the signature process of a wide range of stochastic processes, namely, ``signature SDEs''. The signature SDEs not only include the affine and polynomial processes but also can universally approximate a generic class of stochastic differential equations by linearising the vector fields on the path via the signature. Cuchiero et al. \cite{cuchiero2023signature} reduced this transform to the solution to the Riccati-type ODE system for this general class of stochastic processes from the affine or polynomial perspective.

In contrast to \cite{cuchiero2023signature}, this paper focuses on the time-homogeneous It\^o diffusion process and considers the dependency of the characteristic function of its signature process on time and starting point jointly. The primary contributions of our paper are twofold. Firstly, it is devoted to establishing a Feynman-Kac type theorem to establish the link between a characteristic function of a signature process and a parabolic PDE system. Secondly, we exemplify the application of the PDE theorem using the Brownian motion and Lévy area case, which offers novel proof to recover the classical results on the joint characteristic function of this coupled process.  

\subsection{Characteristic function of the generalized-signature process}
Fix $T>0$, $n\geq 1$ and consider a time-homogeneous and $d$-dimensional It\^o diffusion $X = (X_t)_{t \in [0, T]}$. Let $S^n(X_{[s,t]})$ denote the truncated signature of $X$ defined on the time interval $[s, t]$ up to any fixed degree $n$, where the integral is understood in the Stratonovich sense. We are interested in deriving the characteristic function of the truncated signature $S^n(X_{[0,t]})$ conditional on $X_0 = x$, denoted by $\mathcal{L}_n$, i.e., 
\begin{eqnarray*}
\label{eq:main intro}
    \mathcal{L}_n(t, x ; \cdot) :  T^n(\mathbb{R}^d)\to \mathbb{C},\quad (t, x; \lambda) \mapsto \mathbb{E}[\exp(\mathfrak{i} \langle\lambda, S^n(X_{[0,t]})\rangle) | X_0 = x]. 
\end{eqnarray*}

To tackle this problem, we introduce the generalized-signature process $\mathbb{X}^{n} = (\mathbb{X}_t)_{t \in [0, T]}$ as an intermediate object. It is a generalization of the signature process of $S^{n}(X)$, because it satisfies the same SDE as $S^{n}(X)$. However, the initial condition of $\mathbb{X}^{n}$ at time $0$ can be an arbitrary point in the truncated tensor space, whereas that of $S^{n}(X_{[0, t]})$ is restricted to $\mathbf{1}$. This generalization allows us to write the conditional characteristic function of the generalized-signature process $\mathcal{L}_n$ in terms of $\mathbb{L}_{n}$ explicitly, where $\mathbb{L}_{n}$ is given by
\begin{equation*}
\mathbb{L}_n(t,\mathbbm{x}; \cdot) : T^n(\mathbb{R}^d) \to \mathbb{C},\quad (t,\mathbbm{x}; \lambda) \mapsto \mathbb{E}[\exp(\mathfrak{i} \langle\lambda, \mathbb{X}^n_t - \mathbb{X}^n_0\rangle ) | \mathbb{X}^n_0 = \mathbbm{x}].
\end{equation*}
The relation between $\mathcal{L}_n$ and $\mathbb{L}_n$ is established, and we showed that we can always recover $\mathcal{L}_n$ given $\mathbb{L}_n$, hence we transformed the problem into finding the analytical form of $\mathbb{L}_n$. We established the main PDE Theorem \ref{theorem: general char pde} on $\mathbb{L}_n$. First, we show that under mild regularity condition, $\mathbb{L}_{n}$ satisfies the following PDE system with the initial condition specified in Theorem \ref{theorem: general char pde}. Conversely, the solution to the PDE Equation \eqref{eq:general pde} is proved to be the conditional characteristic function $\mathbb{L}_{n}(\cdot, \cdot; \lambda)$.

\begin{restatable*}{thm}{generalcharpde}
\label{theorem: general char pde}
Fix $n\in \mathbb{N}$. Let $(\mathbb{X}^n_t)_{t\in[0,T]}$ be the generalized-signature process defined in Definition \ref{def:sig-aug process} with drift $\mu_n$ and diffusion $\sigma_n$. Assume that $\mu_n, \sigma_n$ satisfy Condition \ref{conditionVectorField}. Denote by $D_n$ the dimension of $\mathbb{X}^n$. 
\\
Let $\lambda\in T^n(E)$, consider its parametrized linear functional $M_{\lambda}$ and the characteristic function $\mathbb{L}_n(t,\mathbbm{x}; \lambda)$ defined in Definition \ref{def:char generalized-signature process}. Assume that 
$\mathbb{L}_n(t,\mathbbm{x}; \lambda)$ satisfies Condition \ref{condtionESDiffusionFixedTime}. Then, 
\begin{itemize}
    \item $\mathbb{L}_n$ satisfies the following PDE 
\begin{equation}
\label{eq:general pde}
    \begin{split}
        \left(-\frac{\partial }{\partial t}+ A_n\right)\mathbb{L}_n(t,\mathbbm{x}; \lambda) +\mathfrak{i}\left((M_{\lambda}\circ \mu_n)(\mathbbm{x})+\sum_{j=1}^{D_{n-1}}(M_{\lambda}\circ b_n^{(\cdot, j)})(\mathbbm{x})\frac{\partial }{\partial \mathbbm{x}^{(j)}}\right)\mathbb{L}_n(t,\mathbbm{x}; \lambda)
\\
-\frac{1}{2}(M_{\lambda}^{\otimes 2}\circ b_n)(\mathbbm{x})\mathbb{L}_n(t,\mathbbm{x}; \lambda)=0, 
    \end{split}
\end{equation}
with initial condition $\mathbb{L}_n(0, \mathbbm{x}; \lambda) = 1$ for all $(t, \mathbbm{x})\in \R\times T^{n-1}(E)$, $b_n = \sigma_n \sigma_n^T$, $b_n^{(\cdot, j)}$ denotes the $j$-th column of $b_n$ and
\begin{equation*}
     A_nf(\mathbbm{x}_0) = \sum_{i=1}^{D_{n-1}} \mu_n^{(i)}(\mathbbm{x}_0)\frac{\partial f}{\partial \mathbbm{x}^{(i)}}\Bigg \vert_{\substack{\mathbbm{x}=\mathbbm{x}_0}} + \sum_{j_1=1}^{D_{n-1}} \sum_{j_2=1}^{D_{n-1}} \frac{1}{2}b_n^{(j_{1}, j_{2})}(\mathbbm{x}_0) \frac{\partial^2 f}{\partial \mathbbm{x}^{(j_1)}\partial \mathbbm{x}^{(j_2)}}\Bigg \vert_{\substack{\mathbbm{x}=\mathbbm{x}_0}}.
\end{equation*}
\item Let $f_n(t, \mathbbm{x};\lambda)$ be a function that solves the Equation \eqref{eq:general pde} with initial condition $f_n(0, \mathbbm{x};\lambda) = 1$. In addition, assume that first-order derivatives of $f_n$ with respect to $\mathbbm{x}$ satisfy polynomial growth (Definition \ref{definition_linear_growth}). Then $f_n(t, \mathbbm{x};\lambda) = \mathbb{L}_n(t,\mathbbm{x};\lambda)$ for all $(t, \mathbbm{x})\in \mathbb{R}^+ \times T^{n-1}(E)$.
\end{itemize}
\end{restatable*}

Note that $\mathbb{L}_n(\cdot, \cdot; \lambda)$  can be viewed as a family of functions parametized by $\lambda$ mapping from $\mathbb{R}^{+} \times T^n(\mathbb{R}^d)$ to $\mathbb{C}$. In the above theorem, we derive a parabolic PDE to characterize $\mathbb{L}_n(\cdot, \cdot; \lambda)$ for any given $\lambda$.

Our proof follows the martingale approach, sharing the spirit utilized in proving the classical Feynman-Kac theorem and the PDE approach of the expected signature of the diffusion process \cite{ni2012expected_signature}. Our proof is built upon the following two key observations,
\begin{enumerate}
    \item the generalized-signature process is a time homogeneous It\^o diffusion;
    \item the map $F: \mathbb{R}^m \rightarrow  \mathbb{C}; a \mapsto \exp(\mathfrak{i} \langle \lambda, a \rangle)$ is multiplicative. 
\end{enumerate}
The definition of the map $F$ also implies the uniform boundedness of the characteristic function, which allows a weakening of the regularity condition required for the PDE theorem for $\mathbb{L}_n$ and can be applied to any $\lambda$. 

In addition to the martingale approach, we employ a slightly different method by directly applying the PDE theorem of the expected signature in \cite{ni2012expected_signature} (See Appendix \ref{appendix: alternative proof of main theorem}). This approach, however, relies on the essential assumption: the characteristic function $\mathbb{L}_n(t, x; \lambda)$ admits Taylor's power series in terms of $\lambda$ and hence can be written as a linear combination of the expected signature. When the radius of convergence (ROC) of such a power series is finite, we are unable to use this approach to prove the PDEs of the characteristic function for the case $|\lambda|$ exceeds the ROC. Nevertheless, if the ROC condition is met, this Taylor series strategy provides a numerical approximation of $\mathcal{L}_{n}$ via the expected signature, which can be solved recursively.


\subsection{Characteristic function of Brownian motion coupled with the L\'evy area}
The L\'evy area of Brownian motion is a classical mathematical object and plays a crucial role in many domains of stochastic analysis, for example, in numerical approximation for SDEs (\emph{e.g.}, the Milstein method). As demonstrated in works by L\'evy\cite{levy1940area} and Levin and Wildon  \cite{levin2008combinatorial}, two distinct approaches yield the analytical formula of the characteristic function of the Lévy area for the case of 2-dimensional Brownian motion. Cuchiero \cite{LevyAreaAffineSDEs} also regarded the $2$-dimensional case as a special case of the affine process, and hence derived the corresponding Riccati ODE representing the characteristic function. It is more technically involved to derive the analytic formula of the joint process of Brownian motion and the L\'evy area with $d>2$, which is discovered by \cite{helmes1983levy}. The main technique used in \cite{helmes1983levy} is a probabilistic approach based on the Girsanov theorem. 

In this paper, we apply our general PDE theorem to this concrete case. This approach provides completely new proof for the characteristic function of this coupled process and recovers its analytic formula by leveraging the structure theorem of anti-symmetric matrices.   

Concretely, given a $d$-dimensional Brownian motion $W = (W_t^{(1)},\dots ,W_t^{(d)})_{t\in [0, T]}$, consider the process
\begin{equation*}
    L^{(i,j)}_t = \frac{1}{2}\bigg (\int_0^t W_s^{(i)} dW_s^{(j)} - \int_0^t W_s^{(j)} dW_s^{(i)}\bigg ),
\end{equation*} 
for any $1\leq i,j \leq d$. This process depends on the initial point of $W_0$. If $W_0=0$, then $L^{(i,j)}_t$ is the L\'evy area of $i$-th and $j$-th coordinates of a standard Brownian motion from time $0$ to $t$.
In \cite{helmes1983levy}, authors introduced the generalized L\'evy area process associated with the anti-symmetric matrix $\Lambda$ of dimension $d \times d$, denoted by $L^{\Lambda}_t$, i.e.,
\begin{eqnarray*}
L^{\Lambda}_t := \sum_{1\leq j_1< j_2\leq d} \Lambda_{j_1, j_2}L_t^{(j_1, j_2)} = \int_{0}^{t}\langle \Lambda W_s, dW_s \rangle. \label{eq:almost levy area}
\end{eqnarray*} 
Note that $L^{\Lambda}$ can be seen as the linear combination of degree $2$ signature of the Brownian motion. Consequently, we apply the PDE theorem (Theorem \ref{theorem: general char pde}) to this case and solve the conditional characteristic function of the $L^{\Lambda}_t$ explicitly by proposing a suitable Ansatz for the PDE satisfied by the characteristic function. Our proof heavily relies on the structure theorem of the anti-symmetric matrix $\Lambda$, but it does not use any probability theory.

Moreover, by leveraging the translation invariance of Brownian motion, we discover a connection between the conditional characteristic function of $L^{\Lambda}$ and the joint characteristic process. This yields the following theorem, which gives the analytic formula for the joint characteristic function of the coupled process.
\begin{restatable*}{thm}{levychar}
\label{main_theorem}
    Let $W = (W_t^{(1)},\dots ,W_t^{(d)})_{t\in [0, T]}$ be a $d$-dimensional Brownian motion. Let $\mu \in \mathbb{R}^d$ and let $\Lambda$ be a $d$-dimensional anti-symmetric matrix. The joint characteristic function of coupled Brownian motion $W$ and the L\'evy area $L$
    \begin{align*}
        \Psi_{W}(t, \mu, \Lambda) & : \mathbb{R}^+ \times \mathbb{R}^d \times \mathbb{R}^{d\times d} \to \mathbb{C};
        \\
        (t, \mu, \Lambda) & \mapsto \mathbb{E} \bigg[ \exp \left(\mathfrak{i} \sum_{i=1}^d \mu_i W_t^{(i)}+ \mathfrak{i} L_t^{\Lambda} \right)\bigg|W_0 = 0 \bigg],
    \end{align*}
    admits the following formula
    \begin{align*}
        \Psi_W(t,\mu, \Lambda) = & \bigg( \prod_{i=1}^{d_1} \frac{1}{\cosh(\frac{\eta_i}{2}t)} \bigg) 
        \\
        & \exp \bigg( \big[ \sum_{i=1}^{d_1} - \frac{1}{\eta_i} ((O\mu)_{2i-1}^2+(O\mu)_{2i}^2)\tanh(\frac{\eta_i}{2}t) \big] - \frac{1}{2}t \sum_{i=1}^{d_0} (O\mu)^2_{2d_1+i} \bigg),
    \end{align*}
    where $O,\ \eta,\ d_0,\ d_1$ are defined in Lemma \ref{lem_antisym_matrix_decomposition}.
\end{restatable*}

\subsection{The outline of the paper}
In Section 2, we introduce the preliminaries of rough path theory, including the signature and the expected signature of a diffusion process. Besides, we summarize the main results on the PDE of the expected signatures 
 elaborated in \cite{ni2012expected_signature}. In Section 3, we introduce the generalized-signature process and we propose the general methodology, we derive the PDE of the conditional characteristic function $\mathcal{L}_n$ as stated in Theorem \ref{theorem: general char pde}. In Section 4, we focus on solving the PDE associated to the conditional characteristic function of the $L^{\Lambda}$ stated in Equation \eqref{eq:almost levy area}. In Section 5, we first establish the link between $\mathcal{L}_n$ and the characteristic function of the joint process $(W, L^{\Lambda})$. We divide the derivation of the joint characteristic function $\Psi_W(t, \mu, \Lambda)$ into the non-degenerated case and degenerated case and provide complete proof to Theorem \ref{main_theorem}. Section 6 discusses future work to conclude the paper. 
\section{Preliminary}


Throughout this section, we fix $(E, \norm \cdot )$ to be a $d$-dimensional normed space, in which the diffusion process takes values.

\subsection{Time-homogeneous It\^o diffusion}
Consider the probability space $(\Omega, \mathcal{F}, \mathbb{P}_{x})$, let $W = (W_{t})_{t \in [0, + \infty)}$ be a $d'$-dimensional Brownian motion. We assume that the probability space is equipped with the filtration $\{\mathcal{F}_{t}\}_{t \in [0, + \infty]}$ generated by $W$.
We define a time-homogeneous It\^o diffusion process $X$, which is a solution to the following SDE
\begin{eqnarray}\label{eqn_SDE}
dX^{(i)}_{t} = \mu^{(i)}(X_{t})dt + \sum_{j = 1}^{d'} \sigma^{(i)}_j(X_{t}) dW^{(j)}_{t},\quad i = 1,\dots, d,
\end{eqnarray}
where $t \in [0, \infty]$, with vector fields $\mu, \sigma_1, \cdots, \sigma_{d'}$ that satisfy standard regularity conditions (see Definition \ref{definition_linear_growth} and Condition \ref{conditionVectorField}) and the integral is understood in the It\^o sense. Using matrix notation, we write the SDE in equation \eqref{eqn_SDE} as
\begin{equation}
    \label{eqn_SDE_matrix}
    dX_t = \mu(X_t)dt + \sigma(X_t)dW_t,
\end{equation}
with \textit{drift term} $\mu$ and \textit{diffusion term} $\sigma: E \to E\times \mathbb{R}^{d'}$, $\sigma(x) = (\sigma_1(x), \dots, \sigma_{d'}(x))$.
Under the probability measure $\mathbb{P}_{x}$, assume that the process $X$ satisfies the initial condition
\begin{eqnarray*}
\mathbb{P}_{x}(X_{0} = x) = 1.
\end{eqnarray*}
Moreover, it is assumed that the law of the pair $(W, X)$ is uniquely determined. For ease of notation, let $b:\mathbb{R}^d \to \mathbb{R}^d \times \mathbb{R}^d$ denote the $d \times d$ matrix, with elements
\begin{eqnarray*}
b^{(j_{1}, j_{2})}(x) := \sum_{i = 1}^{d'}\sigma^{(j_{1})}_{i}(x)\sigma^{(j_{2})}_{i}(x);\quad  j_{1}, j_{2} \in \{1, 2, \dots, d\}.
\end{eqnarray*}
Using matrix notation, we have $b(x) = \sigma(x)\sigma(x)^T$, we call $b$ the \textit{diffusion matrix}.
\\
\begin{definition}[Infinitesimal generator]
Let $X$ be a time-homogenous It\^o diffusion process. For any function $f \in C^2_{0}(E)$, the infinitesimal generator of $f(X)$ is defined as
\begin{equation*}
    Af(x_0) = \sum_{i=1}^{d} \mu^{(i)}(x_0)\frac{\partial f}{\partial x^{(i)}}\Bigg \vert_{\substack{x=x_0}} + \sum_{j_1=1}^d \sum_{j_2=1}^d \frac{1}{2}b^{(j_{1}, j_{2})}(x_0) \frac{\partial^2 f}{\partial x^{(j_1)}\partial x^{(j_2)}}\Bigg \vert_{\substack{x=x_0}},
\end{equation*}
where $C^2_{0}(E)$ denotes the set of all twice differentiable functions with compact support from $E$ to $\mathbb{R}$.
\end{definition}
In order to discuss the regularity of vector fields, let us introduce the concept of the polynomial growth condition as follows:
\begin{definition}
    \label{definition_linear_growth}
    Let $n, m$ be positive integers, $m\geq 1$. A function $f:\mathbb{R}^+\times E \to \mathbb{R}^n$ is said to satisfy the polynomial growth condition of degree $m$ if and only if there exists $C$ such that
    \begin{equation*}
        \sup_{0\leq t \leq T}|f(t,x)|\leq C(1+|x|^m),
    \end{equation*}
    for any $T>0$.
\end{definition}
In the rest of this section, we assume that the drift term $\mu$ and the diffusion vector fields $\{\sigma_{i}\}_{i = 1}^{d'}$ satisfy the following condition.
\begin{condition}\label{conditionVectorField}
The vector fields $\mu, (\sigma_{j})_{j = 1}^{d'}$ satisfy the globally Lipschitz condition, i.e. there exist a constant $C>0$, such that for every $x, y \in E$,
\begin{eqnarray*}
\vert \mu(x) - \mu(y) \vert + \sum_{i = 1}^{d'}\vert \sigma_{i}(x) - \sigma_{i}(y) \vert \leq C \vert x - y \vert.
\end{eqnarray*}
\end{condition}
\begin{remark}
If the vector fields $\mu, (\sigma_{i})_{i = 1}^{d'}$ satisfy Condition \ref{conditionVectorField}, then they satisfy the linear growth condition as well, i.e. there exists a constant $K$, such that for every $x \in E$,
\begin{eqnarray*}
 \vert\mu(x)\vert + \sum_{i = 1}^{d'}\vert \sigma_{i}(x) \vert \leq K (1+ \vert x \vert).
\end{eqnarray*}
By imposing the linear growth type condition and the Lipschitz type condition, it can be proved that the SDE admits a unique and strong solution, as was first done by K. It\^o(\cite{Ito1946}).
\end{remark}
\subsection{Tensor algebra and signature of a diffusion process}
In this section, we introduce the notion of the signature of a path. We start by defining the space where signatures belong to. The following definition follows from \cite{Terry2007roughpath}.
\begin{definition}[Tensor algebra space]
The space of formal series of tensors of $E$, denoted by $T((E))$, is defined to be the following set of the sequences of tensor powers:
\begin{equation*}
    T((E)) = \{ \mathbf{a} = (a_0,a_1, \cdots) | \forall n \geq 0, a_n \in E^{\otimes n} \}. 
\end{equation*}
It is endowed with two internal operations, an addition and a product, which are defined as follows. Let $\mathbf{a} = (a_0,a_1,\cdots)$ and  $\mathbf{b} = (b_0,b_1,\cdots)$ be two elements of $T((E))$. Then it holds that
\begin{equation*}
    \mathbf{a}+\mathbf{b} = (a_0+b_0,a_1+b_1,\cdots);
\end{equation*}
\begin{equation*}
    \mathbf{a}\otimes \mathbf{b} = (c_0,c_1,\cdots),
\end{equation*}
where for each $n \geq 0$,
\begin{equation*}
    c_n = \sum_{k=0}^n a_k \otimes b_{n-k}.
\end{equation*}
The space $T((E))$ endowed with these two operations and the scalar multiplication $\lambda \mathbf{a} =(\lambda a_0, \lambda a_1, \cdots)$ is a real non-commutative unital algebra, with the unit $\mathbf{1}=(1,0,0,\cdots)$. An element $\mathbf{a} =(a_0,a_1,\cdots)$ of $T((E))$ is invertible if and only if $a_0 \neq 0$. Its inverse is then well-defined and given by the series
\begin{equation*}
    \mathbf{a}^{-1} = \frac{1}{a_0}\sum_{n \geq 0}(\mathbf{1} - \frac{\mathbf{a}}{a_0})^n.
\end{equation*}
\end{definition}
\begin{definition}[Truncated tensor algebra]
    Fix $n \in \mathbb{N}$, the truncated tensor algebra space of degree $n$ is defined as     
    \begin{equation*}
        T^n(E) = \{ \mathbf{a} = (a_0,a_1, \cdots, a_n) | a_i \in E^{\otimes i}, i \in \{1,\dots, n\} \}.
    \end{equation*}
Let 
\begin{equation*}
    D_n \coloneqq \sum_{i=0}^{n} d^i = \frac{d^{n+1}-1}{d-1},
\end{equation*}
    be the dimension of $T^n(E)$ where $d$ is the dimension of $E$. 
\end{definition}
 Suppose $(e_1,\dots, e_d)$ is a basis for $E$, then $E^{\otimes n}$ is spanned by basis vectors $e_I \coloneqq (e_{i_1}\otimes \dots \otimes e_{i_n})$ where $I = (i_1,\dots,i_n) \in \{1,\dots, d\}^n$. Consider the dual space $E^*$ of $E$ with basis $(e^*_1,\dots, e^*_d)$. The basis for $(E^*)^{\otimes n}$ can be easily extended to be the set $e^*_I \coloneqq (e^*_{i_1}\otimes \dots \otimes e^*_{i_n})$ where $I = (i_1,\dots,i_n) \in \{1,\dots, d\}^n$ and we have 
 \begin{equation*}
     \langle e^*_{i_1}\otimes \dots \otimes e^*_{i_n}, e_{j_1}\otimes \dots \otimes e_{j_n}\rangle = \delta_{i_1,j_1},\dots, \delta_{i_n,j_n},
 \end{equation*}
 with $\delta_{i,j} = \mathbbm{1}_{i=j}$ and the inner product is defined in canonical sense. 
\begin{definition}[Projection maps]
     We define three kinds of projection maps. Fix $n\in \mathbb{N}$. We define $\pi^n: T((E)) \to T^n(E)$ to be the canonical projection onto the truncated tensor algebra space, i.e.
\begin{equation*}
    \pi^n(\mathbf{a}) = (a_0,a_1, \dots, a_n),\quad \forall \mathbf{a} = (a_0,a_1, \dots, a_n, \dots)\in T((E)).
\end{equation*}
For computational convenience, we also allow $n$ to be negative, in that case, $\pi^n(\mathbf{a}) = 0$. 
\\
Also, we define the ``degree-wise'' projection $\rho^n: T((E)) \to E^{\otimes n}$ in the way
\begin{equation*}
    \rho^n(\mathbf{a}) = a_n,\quad \forall \mathbf{a}= (a_0,a_1, \dots, a_n, \dots)\in T((E)).
\end{equation*}
Let $I = (i_1,\dots, i_n)\in \{1,\dots, d\}^n$ be a finite length index set of length $n$, we define the coordinate projection map as
 \begin{align*}
     \pi^I & : T((E)) \to \mathbb{R};
     \\
     \mathbf{a} & \mapsto e_I^*(\mathbf{a}).
 \end{align*}
\end{definition}

Now let us introduce the definition of the signature of a time-homogeneous It\^o diffusion process.

\begin{definition}
    Let $X = (X_t)_{t \in [0, T]}$ be a time-homogeneous It\^o diffusion process, which satisfies the SDE in Equation \eqref{eqn_SDE_matrix}. For every $t \in [0, T]$, the signature of $X_{[0, t]}$, denoted by $S(X_{[0,t]})$ is defined to be the solution to the following linear controlled differential equation
    \begin{equation}\label{eq:sde signature}
        dS(X_{[0,t]}) = S(X_{[0,t]}) \otimes (\circ dX_t),\quad S(X_{[0, 0]}) = (1, 0, \dots),
    \end{equation}
    where the integral is understood in the Stratonovich sense. For any $n\in \mathbb{N}$, the truncated signature of $X_{[0,t]}$ of degree $n$ is defined to be $S^n(X_{[0,t]}) \coloneqq \pi^n(S(X_{[0,t]}))$. Throughout the paper, whenever the context is clear, we will omit $X$ and use $S_t,\ S^n_t$ to stand for $S(X_{[0,t]})$ and $S^n(X_{[0,t]})$ respectively.
\end{definition}
    \begin{remark}
        By solving the equation \eqref{eq:sde signature} using Picard iterations \cite{Terry2007roughpath}, we can derive the expression of $S_t$ as the sequence of iterative integrals
        \begin{equation*}
            S_t = (1, X_{[0,t]}^{1}, \dots, X_{[0,t]}^{n} ,\dots) \in T((E)),
        \end{equation*}
        where 
        \begin{equation*}
            X_{[0,t]}^{n} = \underset{0<u_{1} < u_{2}<\dots <u_{n}<t}{\int \dots \int} \circ dX_{t_{1}}\otimes \dots \otimes \circ dX_{t_{n}}.
        \end{equation*}
        provided these iterated integrals are well-defined. In particular, let $I = (i_1,\dots, i_n)\in \{1,\dots, d\}^n$ be a finite length index set of length $n$, we have
\begin{equation*}
\pi^I(S_t) = \underset{0<u_{1} < u_{2}<\dots <u_{n}<t}{\int \dots \int} \circ dX^{(i_1)}_{t_{1}}\otimes \dots \otimes \circ dX^{(i_n)}_{t_{n}}.
\end{equation*}
    \end{remark}

\begin{remark}
    Equivalently, the SDE can be defined in ``degreewise" sense as follows
    \begin{equation*}
    \begin{cases}
    d\rho^n(S_t) = \rho^{n-1}(S_t) \otimes (\circ dX_t),\quad \rho^n(S_t) = 0;
    \\
    S^{0}_t = 1,
    \end{cases}
    \end{equation*}
    for any $n\geq 1$.
\end{remark}

In the following, we summarize the key properties of the signature. For more details, we refer to \cite{Terry2007roughpath}.

\begin{definition}
Let $X:[0,s]\longrightarrow E$ and $Y:[s,t]\longrightarrow E$ be two continuous paths. Their concatenation as the path $X*Y :[0,t] \longrightarrow E$ defined by
\begin{eqnarray*}
(X*Y)_{u} = \begin{cases} X_{u}, &  u \in [0,s], \\ Y_{u}-Y_{s}+X_{s}, & u \in [s, t]. \end{cases}
\end{eqnarray*}
\end{definition}
Denote by $\mathcal{V}(J, E)$ the set of continuous paths from $J$ to $E$ such that the signature transform is well-defined. Also, let $\mathcal{S} \subset T((E))$ be the image of signature transformation applied to $\mathcal{V}(J, E)$. We state here three key properties of signature transformation which reveal the group structure of $\mathcal{S}$.
\begin{theorem}[Invariance under time reparametrization]
Let $X:J\longrightarrow E$ be a continuous path with finite p-variation for some $p<2$. Let $\psi: J\longrightarrow J$ be a reparametrization (surjective, continuous, non-decreasing function). Define $\widetilde{X}: \psi(J)\longrightarrow E$ to be the reparametrized path. Then
\begin{equation*}
    S(X_{J}) = S(\widetilde{X_{J}}).
\end{equation*}
\end{theorem}
\begin{theorem}[Time reversal]
\label{time-reversal}
For a path $X: [a,b] \longrightarrow E$, consider its time-reversal as $\overleftarrow{X}: [a,b] \longrightarrow E$, ${\overleftarrow{X}_t = X_{a+b-t},\ \forall t\in [a,b]}$, then one has
\begin{equation*}
    S(X_{[a,b]})\otimes S(\overleftarrow{X_{[a,b]}}) = \textbf{1},
\end{equation*}
where $\textbf{1}$ is the identity element in $T((E))$.
\end{theorem}
\begin{theorem}[Chen's identity]\label{Chen} Let  $X: [0, s] \rightarrow E$ and $Y: [s, t] \rightarrow E$ be two continuous paths of bounded $1$-variation. Then
\begin{eqnarray*}
S((X*Y)_{[0,t]}) = S(X_{[0,s]}) \otimes S(Y_{[s,t]}).
\end{eqnarray*}
\end{theorem}
Chen’s relation tells us that the range of the signature map acting on bounded variation paths is a group homomorphism from paths space to $\mathcal{S}$. Closeness under multiplication of such range is ensured by Chen’s relation. Existence of inverses is a consequence of Theorem \ref{time-reversal}. Hence, the $\mathcal{S} = (\mathcal{S}, \otimes)$ is actually a group under multiplicative operator $\otimes$, from now on, unless stated explicitly, we assume this group structure on $\mathcal{S}$. We call $\mathcal{S}$ the space of signatures.
\begin{remark}
By embedding it into the tensor algebra space, we can easily check that all the properties of signatures hold for truncated signatures as well.
\end{remark}

\begin{remark}
Although the diffusion processes are defined via the It\^o integration, rough path theory provides an alternative way to define it in the pathwise sense almost surely. More specifically, let $X:= (X)_{t\in [0, T]}$ denote a continuous semi-martingales with the finite quadratic variation 
\begin{eqnarray*}
    [X]_{s, t} = \lim_{|\mathcal{D}| \rightarrow 0} \sum_{i = 1}^{N}|X_{t_i} - X_{t_{i-1}}|^2,
\end{eqnarray*}
where $\mathcal{D} = (t_i)_{i=1}^{N}$ is a finite partition of $[s, t]$. Then we can lift $X$ to the $2$-rough path via $(s, t) \mapsto \pi^{2}(\exp(Y_{s, t}))$, where $Y_{s, t}:= (X_{t} - X_{s}, X_{t} - X_{s}, [X]_{s,t})$ for any $0 \leq s \leq t\leq T$. Consequently, one can make sense of the It\^o integration in the pathwise sense by using this lifting.
\end{remark}

\subsection{PDE of expected signature of a diffusion process} 
\begin{definition}[Expected signature of diffusion process]\label{DefEsSigDiffusionProcess}
Let $X = (X_t)_{t \in [0, T]}$ be an It\^o diffusion process defined by SDE \eqref{eqn_SDE}. Suppose that the signature of $X_{[0, t]}$ is well-defined. The expected signature of $X_{[0,t]}$ starting at point $x$, denoted by $\Phi(t, x)$ is defined as follows:
\begin{align*}
\Phi : \mathbb{R}^+ \times E &\to T((E));
\\
(t,x) & \mapsto \mathbb{E}^{x}[S(X_{[0,t]})],
\end{align*}
where $t \in \mathbb{R}^{+}$. For every $n \in \mathbb{N}$, the $n^{th}$ term of the expected signature of an It\^o diffusion process $X_{[0,t]}$ starting at $x$ is denoted by $\Phi_{n}(t,x)$. Moreover, $\Phi_{0}(t,x) = 1$.
\end{definition}
\begin{remark}
By the definition of the signature, it is known that the signature of any path at time $0$ should be trivial, and thus its expectation is trivial as well, i.e.
\begin{eqnarray*}
\Phi(0,x) = \mathbf{1}.
\end{eqnarray*}
\end{remark}
Here we state a sufficient condition on $\Phi$ for Theorem \ref{eqn_PDE} to hold.
\begin{condition}\label{condtionESDiffusionFixedTime}
$\Phi$ is finite and has continuous partial derivatives with respect to $t$ and continuous second order partial derivatives with respect to $x$, i.e. for any index $I$, $\pi^{I}(\Phi) \in C^{1,2}(\mathbb{R}^+ \times E)$, where $C^{1,2}(\mathbb{R}^+ \times E)$ is the space of continuous functions mapping from $\mathbb{R}^+ \times E$ to $\mathbb{R}$, with continuous partial derivatives $\big( \frac{\partial \Phi}{\partial t} \big)$, $\big( \frac{\partial \Phi}{\partial x_i} \big)_{i=1}^d$, $\big( \frac{\partial^2 \Phi}{\partial x_i \partial x_j} \big)_{i,j=1}^d$.  
\end{condition}
Now, we cite two results from \cite{ni2012expected_signature}, which establish the link between the expected signature of a time-homogeneous diffusion process and a parabolic PDE.
\begin{theorem}\label{eqn_PDE}
Let $X = (X_t)_{t \in [0, T]}$ be an It\^o diffusion process, which is defined by the SDE \eqref{eqn_SDE_matrix} and satisfies Condition \ref{conditionVectorField}. Let $\Phi(t,x)$ be the expected signature of $X_{[0,t]}$ conditioned on $X_0 = x$ (Definition \ref{DefEsSigDiffusionProcess}) and assume it satisfies Condition \ref{condtionESDiffusionFixedTime}. Then $\Phi$ is a function mapping $\mathbb{R}^{+} \times E$ to $T((E))$, which satisfies the following PDE:
\begin{align}\label{pde_expected_signature}
    \left(-\frac{\partial }{\partial t}+ A\right)\Phi(t,x) +\sum_{j=1}^{d}\left(\sum_{j_{1}=1}^{d}b^{(j_{1}, j)}(x)e_{j_{1}}\right) \otimes \frac{\partial \Phi(t,x)}{\partial x_{j}}+\notag \\
\left(\sum_{j = 1}^{d} \mu^{(j)}(x)e_{j}+\frac{1}{2}\sum_{j_{1}=1}^{d}\sum_{j_{2}=1}^{d}b^{(j_{1}, j_{2})}(x)e_{j_{1}}\otimes e_{j_{2}}\right) \otimes \Phi(t,x)=0,
\end{align}
subject to the initial condition:
\begin{eqnarray*}
\Phi(0, x) = \mathbf{1},
\end{eqnarray*}
and the condition:
\begin{eqnarray*}
\Phi_{0}(t, x) = 1,\text{ } \forall (t,x) \in \mathbb{R}^{+} \times E,
\end{eqnarray*}
where $A$ is the infinitesimal generator of $X$, $\mu$ is the drift term of $X$ and $b$ is the diffusion matrix of $X$.
\end{theorem}
\begin{remark}
    By applying the projection map onto the subspace $E^{\otimes n}$ on both sides of Equation \eqref{pde_expected_signature}. We can obtain the recursive relation between $\Phi_n$, $\Phi_{n-1}$, $\Phi_{n-2}$.
\end{remark}
\begin{corollary}\label{corollary_pde}
    Let $X = (X_t)_{t \in [0, T]}$ be an It\^o diffusion process, which is defined by the SDE \eqref{eqn_SDE_matrix} and satisfies Condition \ref{conditionVectorField}. Let $\Phi(t,x)$ be the expected signature of $X_{[0,t]}$ conditioning on $X_0 = x$ (Definition \ref{DefEsSigDiffusionProcess}) and assume it satisfies Condition \ref{condtionESDiffusionFixedTime}. Then, the following PDE is satisfied for all positive integers $n\geq 2$:
\begin{align*}
\left(-\frac{\partial }{\partial t}+ A\right)\Phi_n(t,x)  = &
- \left(\sum_{j = 1}^{d} \mu^{(j)}(x)e_{j}\right) \otimes \Phi_{n-1}(t,x)
\\
& -\sum_{j=1}^{d}\left(\sum_{j_{1}=1}^{d}b^{(j_{1}, j)}(x)e_{j_{1}}\right) \otimes \frac{\partial \Phi_{n-1}(t,x)}{\partial x_{j}}
\\
& - \left(\frac{1}{2}\sum_{j_{1}=1}^{d}\sum_{j_{2}=1}^{d}b^{(j_{1}, j_{2})}(x)e_{j_{1}}\otimes e_{j_{2}}\right) \otimes \Phi_{n-2}(t,x).
\end{align*}
    Moreover, $\Phi_1(t,x)$ should satisfy:
    \begin{eqnarray*}
        \left(-\frac{\partial }{\partial t}+ A\right)\Phi_1(t,x) = -\sum_{j=1}^d \mu^{(j)}(x)e_j,
    \end{eqnarray*}
    where $A$ is the infinitesimal generator of $X$, $\mu$ is the drift term of $X$ and $b$ is the diffusion matrix of $X$.
    \\
    Furthermore, for every integer $n \geq 1$, $\Phi_n(t, x)$ should satisfy the initial condition:
    \begin{equation*}
        \Phi_n(0,x) = 0,
    \end{equation*}
    and the condition:
    \begin{equation*}
        \Phi_{0}(t, x) = 1,\text{ } \forall (t,x) \in \mathbb{R}^{+} \times E.
    \end{equation*}
\end{corollary}
The following theorem gives the converse result of Theorem \ref{eqn_PDE}. It shows under certain regularity conditions, the solution to the PDE in Equation \eqref{pde_expected_signature} is indeed the expected signature.
\begin{theorem}\label{uniquenessPDE}
Let $X = (X_t)_{t \in [0, T]}$ be an It\^o diffusion process, which is defined by the SDE \eqref{eqn_SDE} and satisfies Condition \ref{conditionVectorField}. Let $A$ be the infinitesimal generator of $X$. Let $\Phi(t,x)$ be the expected signature of $X_{[0,t]}$ conditioning on $X_0 = x$ (Definition \ref{DefEsSigDiffusionProcess}). Suppose $\xi:\mathbb{R}^+\times E \to T((E))$ is a solution to the PDE:
    \begin{eqnarray*}
&&\left(-\frac{\partial }{\partial t}+ A\right)\xi(t,x) +\sum_{j=1}^{d}\left(\sum_{j_{1}=1}^{d}b^{(j_{1}, j)}(x)e_{j_{1}}\right) \otimes \frac{\partial \xi(t,x)}{\partial x_{j}}+\\
&&\left(\sum_{j = 1}^{d} \mu^{(j)}(x)e_{j}+\frac{1}{2}\sum_{j_{1}=1}^{d}\sum_{j_{2}=1}^{d}b^{(j_{1}, j_{2})}(x)e_{j_{1}}\otimes e_{j_{2}}\right) \otimes \xi(t,x)=0,
\end{eqnarray*}
subject to the initial condition:
\begin{eqnarray*}
\xi(0, x) = \mathbf{1},
\end{eqnarray*}
and the condition:
\begin{eqnarray*}
\xi_{0}(t, x) = 1,\text{ } \forall (t,x) \in \mathbb{R}^{+} \times E.
\end{eqnarray*}
Moreover, suppose $\xi(t,x)$ and the first derivative of $\xi(t,x)$ satisfy the polynomial growth condition. Then $\xi(t,x)$ is the expected signature of $X$ over the time interval $[0,t]$ with the starting point $x$, i.e.
\begin{equation*}
    \xi(t,x) = \Phi(t,x),
\end{equation*}
where $(t,x)\in \mathbb{R}^+\times E$.
\end{theorem}
\subsection{Basics of Partial Differential Equations}
In this section, we present some basic notions and results on parabolic PDEs for a detailed explanation, refer to \cite{Friedman1964pde}. Consider the PDE of the following form
\begin{equation}\label{parabolic_pde}
    -\frac{\partial f}{\partial t} + \sum_{i,j=1}^d a_{ij}(t,x)\frac{\partial^2 f}{\partial x_i \partial x_j} + \sum_{i=1}^d b_i(t,x) \frac{\partial f}{\partial x_i} + c(t,x) f = d(t,x),
\end{equation}
for $(t,x)\in (0,T)\times \Omega$, where $\Omega$ is an open, path-wise connected subset of $E$ and $f\in C^{1,2}([0,T]\times E)$, $f(t,x)\in \mathbb{C}$. Let 
\begin{eqnarray*}
    && D = (0,T] \times \Omega;
    \\
    && \Sigma = \partial \Omega \times [0,T] \cup \Omega \times \{0\},
\end{eqnarray*}
where $ \partial \Omega$ denotes the boundary of $\Omega$. We assume the coefficients functions $a_{ij},\ b_i,\ c$ are all bounded on closure of $D$ and $a_{ij} = a_{ji}$.
\begin{definition}
    The PDE in equation \eqref{parabolic_pde} is uniformly parabolic if there exists $\lambda > 0$ such that
    \begin{equation*}
        \sum_{i,j=1}^d a_{ij}(x,t)u_i u_j \geq \lambda \norm{u}^2,
    \end{equation*}
    for all $(t,x)\in \Bar{D}$ and $u\in E$.
\end{definition}
Let us recall the celebrated Feynman-Kac formula, which establishes a link between parabolic PDEs and stochastic processes. The formula will be used to prove the uniqueness of the solution to the PDE we will derive in the next section.
\begin{theorem}
    \label{feynman-Kac}
    Let $X = (X_t)_{t \in [0, T]}$ be an It\^o diffusion process defined in Equation \eqref{eqn_SDE_matrix} which satisfies condition \ref{conditionVectorField}. Let $A$ be the infinitesimal generator of $X$. Let $T >0$. Suppose that $F \in C^{1,2}([0,T]\times E, \mathbb{R})$ satisfying
    \begin{equation*}
        -\frac{\partial F(t,x)}{\partial t} = AF(t,x) - V(t,x)F(t,x) + f(t,x);
    \end{equation*}
    for all $x\in E$ and $t \in [0,T]$ subject to the terminal condition
    \begin{equation*}
        F(T,x) = \psi(x),\ \forall x \in E,
    \end{equation*}
    where $V, f$ and $\psi$ are real-valued functions that satisfy the polynomial growth condition. Then $F$ admits the solution
    \begin{equation*}
        F(t,x) = \mathbb{E} \Bigg[ \psi(X_T)Z(t, X_t) + \int_t^T f(s,X_s)Z(s, X_s) ds \bigg\vert X_t=x\Bigg],
    \end{equation*}
    where $t \in [0,T]$ and
    \begin{equation*}
        Z(t,x) = \exp\bigg(-\int_t^T V(s, x)ds\bigg).
    \end{equation*}
    Furthermore, the solution $F$ is unique.
\end{theorem}
\begin{proof}
The proof can be found in \cite{Karatzas1991stocal}.
\end{proof}
\section{General methodology}
\label{sec:general methodology}
We propose in this section a general approach to characterize the law of the truncated signature of an It\^o diffusion in a closed time interval. Let $(E, \norm \cdot )$ be a $d$-dimensional Banach space, and let $X_t$ be a $E$-valued time-homogeneous It\^o diffusion. For any $n\in \mathbb{N}$, the conditional characteristic function of $S^n(X_{[0,t]})$ is given by
\begin{align}
    \mathcal{L}_n(t,x; \cdot) : T^n(E) \to \mathbb{C},\quad (t, x; \lambda) \mapsto \mathbb{E}[\exp(\mathfrak{i} \langle\lambda, S^n(X_{[0,t]})\rangle) | X_0 = x]. \label{eq:conditional char enhanced}
\end{align}
We demonstrate in this section that $\mathcal{L}_n$ can be well understood via the parabolic PDE in Equation \eqref{eq:general pde}. We consider an intermediate object, the so-called generalized-signature process.



\begin{definition}[Generalized-signature process]\label{def:sig-aug process}
Let $X = (X_t)_{t \in [0, T]}$ denote an $E$-valued time-homogeneous diffusion process with the drift $\mu$ and the diffusion $\sigma$. The generalized-signature process of $X_{[0,t]}$, denoted by $(\mathbb{X}_t)_{t \in [0, T]}$, is defined as a time-homogeneous diffusion process taking values in $T((E))$, which satisfies the following SDE,
\begin{equation*}
\begin{cases}
d\mathbb{X}_t = \mathbb{X}_t \otimes \big( (\mu(\rho^{1}(\mathbb{X}_t))dt+ \sigma(\rho^{1}(\mathbb{X}_t ))\circ dW_t )\big),\quad \mathbb{X}_0 = \mathbbm{x};
\\
\mathbb{X}^{0}_t = 1,
\end{cases}
\end{equation*}
where $\mathbbm{x} \in T(E)$ and the integral $\circ$ is understood in the Stratonovich's sense. For any $n\geq 1$, we define $\mathbb{X}^n$ to be $\pi^n(\mathbb{X})$, the truncated generalized-signature of degree $n$.
\end{definition}

Equivalently, we can express the dynamics of $\mathbb{X}$ in the It\^o sense, see the lemma below.
\begin{lemma}\label{lemma:enhanced diffusion}
The generalized-signature process $(\mathbb{X}_t)_{t \in [0, T]}$ defined in Definition \ref{def:sig-aug process} admits the following SDE 
\begin{align*}
    d \mathbb{X}_t & = \mu_{\mathbb{X}}(\mathbb{X}_t)dt + \sigma_{\mathbb{X}}(\mathbb{X}_t)dW_t,\quad \mathbb{X}_0 = \mathbbm{x}
\end{align*}
where $\mathbbm{x} \in T((E))$ and the integral is understood in It\^o's sense and
\begin{align*}
    \mu_{\mathbb{X}}: T((E)) \to T((E)),\quad \mu_{\mathbb{X}}(\mathbbm{x}) & =  \left( \mathbbm{x} \otimes \mu(\rho^1(\mathbbm{x})) + \frac{1}{2} \mathbbm{x} \otimes (\sigma \otimes \sigma^T)(\rho^1(\mathbbm{x})) \right);
    \\
    \sigma_{\mathbb{X}}: T((E)) \to  T((E)) \times \mathbb{R}^{d'},\quad \sigma_{\mathbb{X}}(\mathbbm{x})w & = \left(\mathbbm{x} \otimes (\sigma(\rho^1(\mathbbm{x}))w) \right).
\end{align*}
Equivalently speaking, for any $n\geq 1$, the truncated signature-augmented process $(\mathbb{X}^n_t)_{t\in [0, T]}$ satisfies the following SDE
\begin{eqnarray*}
    d \mathbb{X}^n_t & = \mu_n(\mathbb{X}^n_t)dt + \sigma_n(\mathbb{X}^n_t)dW_t,
\end{eqnarray*}
with drift
\begin{align*}
    \mu_{n} & : T^n(E) \to T^n(E),\quad \mu_{n} = (\mu^{(1)} \oplus \dots \oplus \mu^{(n)});
    \\
    \mu^{(i)} & :  T^n(E) \to E^{\otimes i},\quad \mu^{(i)}(\mathbbm{x}) = \left( \rho^{n-1}(\mathbbm{x}) \otimes \mu(\rho^1(\mathbbm{x})) + \frac{1}{2} \rho^{n-2}(\mathbbm{x}) \otimes (\sigma \otimes \sigma^T)(\rho^1(\mathbbm{x})) \right),
\end{align*}
and diffusion
\begin{align*}
    \sigma_{n} & : T^n(E) \to T^n(E),\quad \sigma_{n} = (\sigma^{(1)} \oplus \dots \oplus \sigma^{(n)});
    \\
    \sigma^{(i)} & :  T^n(E) \to E^{\otimes i} \times \mathbb{R}^{d'},\quad \sigma^{(i)}(\mathbbm{x})w = \left( \rho^{n-1}(\mathbbm{x}) \otimes (\sigma(\rho^1(\mathbbm{x}))w) \right).
\end{align*}

\end{lemma}
\begin{proof}
The proof follows directly from the It\^o-Stratonovich correction. Indeed, we can rewrite the SDE of $\mathbb{X}$ as follows
\begin{align}
    d\mathbb{X}_t & = \mathbb{X}_t \otimes \big(\mu(\rho_{1}(\mathbb{X}_t))dt+ \sigma(\rho_{1}(\mathbb{X}_t ))\circ dW_t \big)\notag
    \\
    & = \bigg [\mathbb{X}_t \otimes \mu(\rho_{1}(\mathbb{X}_t)) + \frac{1}{2} \mathbb{X}_t \otimes (\sigma \otimes \sigma^T)(\rho_{1}(\mathbb{X}_t)) \bigg ]dt + \mathbb{X}_t \otimes (\sigma(\rho_{1}(\mathbb{X}_t)) dW_{t})\label{eq:detailed sde}
    \\
    & = \mu_{\mathbb{X}}(\mathbb{X}_t) dt + \sigma_{\mathbb{X}}(\mathbb{X}_t) dW_t,\notag
\end{align}
where the integral is now understood in It\^o's sense and
\begin{align*}
    \mu_{\mathbb{X}}(\mathbbm{x}) & \coloneqq \mathbbm{x}\otimes \mu(\mathbbm{x}_1) + \frac{1}{2}\mathbbm{x}\otimes (\sigma \otimes \sigma^T)(\mathbbm{x}_1),\quad \mathbbm{x} = (1, \mathbbm{x}_1, \mathbbm{x}_2, \dots )\in T((E));
    \\
    \sigma_{\mathbb{X}}(\mathbbm{x})w & \coloneqq  \mathbbm{x} \otimes (\sigma(\mathbbm{x}_1)w) ,\quad \mathbbm{x} = (1, \mathbbm{x}_1, \mathbbm{x}_2, \dots )\in T((E)).
\end{align*}
Similarly, for any $n\geq 1$, by projecting both sides of Equation \eqref{eq:detailed sde} onto the $n$-th degree of the tensor algebra space we have
\begin{align*}
    d\rho^n(\mathbb{X}_t) & = \bigg [\rho^{n-1}(\mathbb{X}_t) \otimes \mu(\rho^{1}(\mathbb{X}_t)) + \frac{1}{2} \rho^{n-2}(\mathbb{X}_t) \otimes (\sigma \otimes \sigma^T)(\rho^{1}(\mathbb{X}_t)) \bigg ]dt + \rho^{n-1}(\mathbb{X}_t) \otimes (\sigma(\rho^{1}(\mathbb{X}_t)) dW_{t})
    \\
    & = \mu^{(n)}(\mathbb{X}^n_t) dt + \sigma^{(n)}(\mathbb{X}^n_t) dW_t,
\end{align*}
with 
\begin{align*}
    \mu^{(n)}(\mathbbm{x}) & \coloneqq \mathbbm{x}_{n-1} \otimes \mu(\mathbbm{x}_1) + \frac{1}{2} \mathbbm{x}_{n-2} \otimes (\sigma \otimes \sigma^T)(\mathbbm{x}_1),\quad \mathbbm{x} = (1, \mathbbm{x}_1, \mathbbm{x}_2,\dots, \mathbbm{x}_n)\in T^n(E);
    \\
    \sigma^{(n)}(\mathbbm{x})w & \coloneqq \left(\mathbbm{x}_{n-1} \otimes (\sigma(\mathbbm{x}_1)w) \right),\quad \mathbbm{x} = (1, \mathbbm{x}_1, \mathbbm{x}_2,\dots, \mathbbm{x}_n)\in T^n(E).
\end{align*}
Finally since $\mathbb{X}^n_t = (1, \rho^{(1)}(\mathbb{X}_t), \dots, \rho^{(n)}(\mathbb{X}_t))$, we obtain
\begin{align*}
    \mu_{n} & = (\mu^{(1)} \oplus \dots \oplus \mu^{(n)});
    \\
    \sigma_{n} & = (\sigma^{(1)} \oplus \dots \oplus \sigma^{(n)}).
\end{align*}
\end{proof}

In the following lemma, we show the relationship between $S(X_{[0,t]})$ and $\mathbb{X}_t$.

\begin{lemma}\label{lemma: Xt equals to St}
Let $X = (X_t)_{t \in [0, T]}$ denote a $E$-valued time-homogeneous It\^o diffusion process with drift $\mu$, diffusion $\sigma$ and the starting point $X_0 = x$. If $\mathbb{X}$ is the generalized-signature process of $X_{[0,t]}$ with starting point $\mathbb{X}_0 = (1, x, 0, \dots)$ and $\mu_{\mathbb{X}}$, $\sigma_{\mathbb{X}}$ satisfy Condition \ref{conditionVectorField} then for any $n\geq 1$
\begin{equation}\label{eq:relation Xt and St}
    \pi^{n}(\mathbb{X}_t) = \pi^{n}(S(X_{[0, t]})) + x \otimes \pi^{n-1}(S(X_{[0, t]})).
\end{equation}
\end{lemma}

\begin{proof}
    For ease of the notation, we use $S_t$ to stand for $S(X_{[0, t]})$. We prove the statement by induction. For $n=1$, since $\mathbb{X}_0 = (1, x, 0, \dots)$ by definition we have $d\pi^1(\mathbb{X}_t) = 1 \otimes dX_t$, therefore $\pi^1(\mathbb{X}_t) = (1, X_t)$. On the other hand, $d\pi^1(S_t) = dX_t$ implies that $\rho^1(S_t) = (1, X_t - x)$, hence Equation \eqref{eq:relation Xt and St} holds.
    \\
    Now assume that Equation \eqref{eq:relation Xt and St} holds for $n$. Then
    \begin{equation*}
        d\pi^{n+1}(S_t) = \pi^{n}(S_t)\otimes (\circ dX_t) = \pi^{n}(\mathbb{X}_t)\otimes (\circ dX_t) - x\otimes \pi^{n-1}(S_t)\otimes (\circ dX_t). 
    \end{equation*}
    Taking the integration, we have
    \begin{equation*}
        \pi^{n+1}(S_t) = \pi^{n+1}(\mathbb{X}_t) -  x\otimes \pi^{n}(S_t),
    \end{equation*}
    where we used Condition \ref{conditionVectorField} to ensure the uniqueness of the solution, and we completed the proof. 
\end{proof}

We introduce here the characteristic function of the generalized-signature process. For notation convenience, in the rest of the paper, we denote by $M_{\lambda}: T^n(E)\to \mathbb{R}, x \mapsto \langle \lambda, x \rangle$ the linear functional parametrized by $\lambda$ for every $\lambda \in  T^n(E)$.
\begin{definition}(Characteristic function of generalized-signature process)
    \label{def:char generalized-signature process}
    Let $n \geq 1$, given any $(t, \mathbbm{x}) \in \mathbb{R}^+
    \times T^n(E)$ and generalized-signature process $\mathbb{X}^n$, we define
    \begin{equation*}
    \mathbb{L}_n(t, \mathbbm{x}, \cdot) : T^n(\mathbb{E}) \to \mathbb{C},\quad (t,\mathbbm{x}; \lambda) \mapsto \mathbb{E}[\exp(\mathfrak{i} M_{\lambda}(\mathbb{X}^n_t - \mathbb{X}^n_0)) | \mathbb{X}^n_0 = \mathbbm{x}],
    \end{equation*}
    to be the characteristic function of $\mathbb{X}$.
\end{definition}
Note that $\mathbb{L}_n(\cdot, \cdot; \lambda)$ can be viewed as a family of functions parametized by $\lambda$ mapping from $\mathbb{R}^{+} \times T^n(E)$ to $\mathbb{C}$. We state here the condition on $\mathbb{L}_n$ for \ref{theorem: general char pde} to hold.
\begin{condition}\label{condtionCharDiffusionFixedTime}
For every $\lambda \in T^n(E)$, $\mathbb{L}_n(t, \mathbbm{x}; \lambda)$ has continuous partial derivatives with respect to $t$ and continuous second order partial derivatives with respect to $\mathbbm{x}$.
\end{condition}
The next lemma explores the dependency of $\mathbb{L}_n$ on the starting point of $\mathbb{X}^n$.
\begin{lemma}\label{lemma: independence char}
    Let $n\geq 1$, for any $(t,\mathbbm{x}) \in \mathbb{R}^+ \times T^n(E)$, we have $\mathbb{L}_n(t, \mathbbm{x}, \cdot) = F_n(t, \pi^{n-1}(\mathbbm{x}), \cdot)$ where
    \begin{equation*}
        F_n(t, \pi^{n-1}(\mathbbm{x}), \cdot) : T^{n-1}(\mathbb{E}) \to \mathbb{C},\quad (t,\mathbbm{x}; \lambda) \mapsto \mathbb{E}[\exp(\mathfrak{i} M_{\lambda}(\mathbb{X}^n_t - \mathbb{X}^n_0)) | \pi^{n-1}(\mathbb{X}^n_0) = \pi^{n-1}(\mathbbm{x})],
    \end{equation*}
    i.e. $\mathbb{L}_n$ is independent of $\rho^n(\mathbbm{x})$.
\end{lemma}

\begin{proof}
    Let $\mathbbm{x}$ be the starting point of $\mathbb{X}^n$. From the definition of $\mathbb{X}^n$ we know that $\mathbb{X}^n_t$ depends linearly on $\rho^n(\mathbbm{x})$ furthermore, the dependency can be written as $\mathbb{X}^n_t = f(\pi^{n-1}(\mathbbm{x}), \rho^1(\mathbb{X}^n_{[0,t]})) + \rho^n(\mathbbm{x})$ for some function $f$. hence $\mathbb{X}^n_t - \mathbb{X}^n_0$ only depends on $\pi^{n-1}(\mathbbm{x})$ as starting point. Since $\mathbb{L}_n(t, \mathbbm{x}, \lambda)$ only depends on the increment $\mathbb{X}^n_t - \mathbb{X}^n_0$, the result follows inmediately.
\end{proof}

\begin{remark}
From now on, for any $ (t, \mathbbm{x}) \in \mathbb{R}^+ \times T^{n-1}(E)$, we consider
\begin{equation}
    \label{eq:char generalized-signature process}
    \mathbb{L}_n(t, \mathbbm{x}; \cdot) : T^{n}(E) \to \mathbb{C},\quad (t,\mathbbm{x}; \lambda) \mapsto \mathbb{E}[\exp(\mathfrak{i} M_{\lambda}(\mathbb{X}^n_t - \mathbb{X}^n_0)) | \pi^{n-1}(\mathbb{X}^n_0) = \mathbbm{x}],
    \end{equation}
to emphazise the dependency on the first $n-1$ degrees of $\mathbb{X}^n$.
    
\end{remark}

The relation between $\mathbb{L}_n$ and $\mathcal{L}_n$ is shown in Lemma \ref{Corollary_L_n} using the relation between $\mathbb{X}^n$ and $S^n$ stated in Lemma \ref{lemma: Xt equals to St}. As a consequence, the objective of deriving $\mathcal{L}_n$ can be reformulated as the problem of solving $\mathbb{L}_n$ and identifying the corresponding PDE, which is satisfied by $\mathbb{L}_n$. 
\begin{lemma}[Link between $\mathcal{L}_n$  and $\mathbb{L}_n$]\label{Corollary_L_n}
Fix $n\geq 1$, let $\mathcal{L}_n$ and $\mathbb{L}_n$ be defined in Equation \eqref{eq:conditional char enhanced} and Definition \ref{def:char generalized-signature process}. For every $(t, x)\in \mathbb{R}^{+} \times E$ and $\lambda \in T^n(E)$, there exists $\Tilde{\lambda} \in T^n(E)$ which depends on $x$, such that
\begin{equation*}
    \mathcal{L}_n(t, x; \lambda) = \mathbb{L}_n(t, (1,x,0,\dots,0);\Tilde{\lambda})\exp{(iM_{\Tilde{\lambda}}(1,x,0,\dots,0))}.
\end{equation*}
Furthermore, if we let $\lambda = \sum_{j = 0}^n \sum_{\vert I \vert = j} \lambda_I e_I$, then $\Tilde{\lambda}$ is given in a recursive form:
\begin{align}
         \Tilde{\lambda}_I & = \lambda_I,\quad \text{if }\vert I\vert = n; \label{eq:recursive_1}
        \\
        \Tilde{\lambda}_I & = \lambda_I - \sum_{j=1}^d x_{j} \Tilde{\lambda}_{\{j\}\cup I},\quad \text{if }1 \leq \vert I\vert < n; \label{eq:recursive_2}
        \\
        \Tilde{\lambda}_{\emptyset} & = \lambda_{\emptyset}, \label{eq:recursive_3}
    \end{align}
     where $\{j\}\cup I$ is defined as the concatenation of indices, i.e. for any $I = (i_1,\dots i_n)$, $\{j\}\cup I = (j, i_1,\dots i_n)$.
\end{lemma}
\begin{proof}
    Fix $\lambda \in T^n(E)$, we want to find $\Tilde{\lambda}$ such that $M_{\lambda}(S^n_t) = M_{\Tilde{\lambda}}(\mathbb{X}^n_t)$. By Lemma \ref{lemma: Xt equals to St}, we have that 
    \begin{equation*}
        \mathbb{X}^n_t = S^n_t + (0, x, x \otimes \rho^1(S_t),\dots, x\otimes \rho^{n-1}(S_t)),
    \end{equation*}
    We are going to construct $\Tilde{\lambda}$ based on $\lambda$. Note first that, $\lambda = \sum_{j = 0}^n \sum_{\vert I \vert = j} \lambda_I e_I$, therefore, we have
    \begin{align*}
        M_{\lambda}(S^n_t) & = \lambda_{\emptyset} + \sum_{j = 1}^n \sum_{\vert I \vert = j} \lambda_I \pi^I (S^n_t);
        \\
        M_{\Tilde{\lambda}}(\mathbb{X}^n_t) & = \Tilde{\lambda}_{\emptyset} + \sum_{j = 1}^n \sum_{\vert I \vert = j} \Tilde{\lambda}_I (\pi^I (S^n_t) + x_{I_1} \pi^{I'}(S^{n-1}_t)),
    \end{align*}
    where $I_1$ denotes the first coordinate of $I$ and $I' = I \setminus I_1$. Hence, by comparing the coefficients of $\pi^I(S^n_t)$ for all indices $I$ with $|I|\leq n$, we get that $\Tilde{\lambda}$ must satisfy the system of equations \eqref{eq:recursive_1}, \eqref{eq:recursive_2} and \eqref{eq:recursive_3}.
    \\
    Observing the system of equations we notice that $\{\Tilde{\lambda}_I\}$ depends only on $\{\Tilde{\lambda}_J\}$ and $x$ where $\vert I \vert = \vert J \vert - 1$, allowing us to solve it recursively, starting with $\vert I \vert = n$. Therefore, we are able to construct a specific parameter tensor $\Tilde{\lambda}$. Finally,
    \begin{align*}
        \mathcal{L}_n(t, x; \lambda) & = \mathbb{E}[\exp(\mathfrak{i} M_{\Tilde{\lambda}}(\mathbb{X}^n_t)\vert \mathbb{X}^n_0 = (1, x, 0, \dots, 0)]
        \\
        & = \mathbb{E}[\exp(\mathfrak{i} M_{\Tilde{\lambda}}(\mathbb{X}^n_t - \mathbb{X}^n_0)\exp(\mathfrak{i}M_{\Tilde{\lambda}}(\mathbb{X}^n_0))\vert \mathbb{X}^n_0 = (1, x, 0, \dots, 0)]
        \\
        & = \mathbb{L}_n(t,(1, x, 0, \dots, 0; \Tilde{\lambda}) \exp(\mathfrak{i}M_{\Tilde{\lambda}}((1, x, 0, \dots, 0))).
    \end{align*}
\end{proof}

\subsection{PDE representation of the characteristic function}
In this subsection, we provide in Theorem \ref{theorem: general char pde} the PDE representation of $\mathbb{L}_n$, the conditional characteristic function of the generalized-signature process, and hence the representation of $\mathcal{L}_n$ stated in Equation \eqref{eq:conditional char enhanced}. We adopt the same ideology as Feynman-Kac (Theorem \ref{feynman-Kac}) by constructing a martingale in the complex plane, taking advantage of the uniform boundedness of the characteristic function. 

\generalcharpde

\begin{proof}
Fix $\lambda\in T^n(E)$, we consider $\mathbb{L}_n(t,\mathbbm{x}; \lambda)$ which satisfies the initial condition 
\\${\mathbb{L}_n(0, \mathbbm{x}; \lambda) = 1}$. 
Fix $t > 0$, consider now the process
    \begin{equation*}
        Y_s = \mathbb{E} [\exp(\mathfrak{i} \langle \lambda, \mathbb{X}^n_t \rangle) \vert \mathcal{F}_s].
    \end{equation*}
    We prove that $Y$ is a martingale, then by computing the drift process of $Y$ we obtain the PDE in the statement. 
    \\
    Since $\exp(\mathfrak{i} \langle \lambda, \mathbb{X}^n_t \rangle)$ is uniformly bounded for any $\lambda$, thus $Y$ is well-defined. Also, we have
    \begin{align*}
        \mathbb{E}[|Y_s|] & = \mathbb{E}[|\mathbb{E} [\exp(\mathfrak{i} \langle \lambda, \mathbb{X}^n_t \rangle) \vert \mathcal{F}_s]|]
        \\
        & \leq \mathbb{E}[\mathbb{E} [|\exp(\mathfrak{i} \langle \lambda, \mathbb{X}^n_t \rangle)| \vert \mathcal{F}_s]|] < \infty,
    \end{align*}
    hence we conclude that $Y$ is a martingale. Next, using the multiplicative property of exponential, we have that
    \begin{align*}
        Y_s & = \mathbb{E} [\exp(\mathfrak{i} \langle \lambda, \mathbb{X}^n_t \rangle) \vert \mathcal{F}_s] 
        \\
        & = \exp(\mathfrak{i}\langle \lambda, \mathbb{X}^n_s \rangle)\mathbb{E}[\exp(\mathfrak{i}\langle \lambda, \mathbb{X}^n_{t} - \mathbb{X}^n_{s} \rangle )\vert \mathbb{X}^n_s)] && \text{$\mathbb{X}^n_s$ is $\mathcal{F}_s$-measurable, Markov prop. of $\mathbb{X}^n$}
        \\
        & = \exp(\mathfrak{i}\langle \lambda, \mathbb{X}^n_s \rangle)\mathbb{L}_n(t-s, \mathbb{X}^n_s; \lambda)&& \text{$\mathbb{X}^n_s$ is time-homogeneous}.
    \end{align*}
We denote by $Z_s = \exp(\mathfrak{i}\langle \lambda, \mathbb{X}^n_s \rangle)$ and $F_n(s, \mathbb{X}^n_s; \lambda) = \mathbb{L}_n(t-s, \mathbb{X}^n_s; \lambda)$ use It\^o's product rule, we obtain
\begin{equation}
\label{eq:ito product}
    dY_s = dZ_s F_n(s, \mathbb{X}^n_s; \lambda) + Z_s dF_n(s, \mathbb{X}^n_s; \lambda) + [dZ_s,dF_n(s, \mathbb{X}^n_s; \lambda)].
\end{equation}
By It\^o's formula, it is easy to obtain the SDE driven by $Z$ and $F_n$. By direct computation, we have
\begin{align}
    dZ_s & = Z_s \bigg(\mathfrak{i} \sum_{j=1}^{D_{n-1}} \lambda_j \mu^{(j)}_n(\mathbb{X}^n_s) - \frac{1}{2}\sum_{k_1=1}^{D_{n-1}}\sum_{k_2=1}^{D_{n-1}} \lambda_{k_1}\lambda_{k_2}b^{(k_1,k_2)}_n(\mathbb{X}^n_s))\bigg)ds \notag
    \\
    & + \mathfrak{i}Z_s \sum_{k_1=1}^{D_{n-1}} \bigg(\sum_{k_2=1}^{D_{n-1}} \lambda_{k_2} \sigma_n^{(k_2, k_1)}(\mathbb{X}^n_s) dW_s\bigg)^{(k_1)}; \label{eq:SDE dZ}
    \\
    dF_n(s, \mathbb{X}^n_s; \lambda) & = \bigg(\frac{\partial}{\partial s} + A_n\bigg)F_n(s, \mathbb{X}^n_s; \lambda) ds \notag
    \\
    & + 
    \bigg( \sum_{k_1=1}^{D_{n-1}} \sum_{k_2=1}^{D_{n-1}}  \sigma_n^{(k_2,k_1)}(\mathbb{X}^n_s) \frac{\partial}{\partial \mathbbm{x}^{(k_2)}} \bigg) F_n(s, \mathbb{X}^n_s; \lambda) dW^{(k_1)}_s.\label{eq:SDE Fn}
\end{align}    
Substituting into the Equation \eqref{eq:ito product}, we have
\begin{align}
    dY_s & = Z_s D(s, \mathbb{X}^n_s;\lambda) ds + Z_s\Sigma((s, \mathbb{X}^n_s;\lambda)) dW_s;\notag
    \\
    D(s, \mathbb{X}^n_s;\lambda) & = \bigg(\frac{\partial}{\partial s} + A_n\bigg)F_n(s, \mathbb{X}^n_s; \lambda) \notag
    \\
    & + \mathfrak{i} \sum_{j_1=1}^{D_{n-1}} \sum_{j_2=1}^{D_{n-1}} \sum_{k=1}^{D_{n-1}} \lambda_{j_1} \sigma_n^{(j_1,k)}(\mathbb{X}^n_s)\sigma_n^{(j_2,k)}(\mathbb{X}^n_s) \frac{\partial F_n(s, \mathbb{X}^n_s; \lambda)}{\partial \mathbbm{x}^{(j_2)}}\notag
    \\
    & + \mathfrak{i} \sum_{j=1}^{D_{n-1}} \lambda_j \mu^{(j)}_n(\mathbb{X}^n_s) - \frac{1}{2}\sum_{k_1=1}^{D_{n-1}}\sum_{k_2=1}^{D_{n-1}} \lambda_{k_1}\lambda_{k_2}b^{(k_1,k_2)}_n(\mathbb{X}^n_s);\label{eq: martingale drift}
    \\
    \Sigma(s, \mathbb{X}^n_s;\lambda) & = \sum_{k_2=1}^{D_{n-1}} \lambda_{k_2} \sigma_n^{(k_2, \cdot)}(\mathbb{X}^n_s) + \sum_{k_2=1}^{D_{n-1}}  \sigma_n^{(k_2,\cdot)}(\mathbb{X}^n_s) \frac{\partial }{\partial \mathbbm{x}^{(k_2)}}F_n(s, \mathbb{X}^n_s;\lambda).\label{eq: martingale diffusion}
\end{align}
Since $Z$ is non-zero and $Y$ is a martingale, we must have that $D(s, \mathbbm{x}; \lambda) = 0$ for any $s\in [0, t]$. Rearranging the terms for $D(s, \mathbbm{x}; \lambda)$ we obtain
\begin{align*}
    & \sum_{j=1}^{D_{n-1}} \lambda_j \mu^{(j)}_n(\mathbbm{x}) = (M_{\lambda} \circ \mu_n)(\mathbbm{x});
    \\
    & \sum_{k_1=1}^{D_{n-1}}\sum_{k_2=1}^{D_{n-1}} \lambda_{k_1}\lambda_{k_2}b^{(k_1,k_2)}_n(\mathbbm{x}) = (M_{\lambda}^{\otimes 2}\circ b_n)(\mathbbm{x});
    \\
    & \sum_{j_1=1}^{D_{n-1}} \sum_{j_2=1}^{D_{n-1}} \sum_{k=1}^{D_{n-1}} \lambda_{j_1} \sigma_n^{(j_1,k)}(\mathbbm{x})\sigma_n^{(j_2,k)}(\mathbbm{x}) \frac{\partial F_n}{\partial \mathbbm{x}^{(j_2)}} = \left(\sum_{j=1}^{D_{n-1}}(M_{\lambda}\circ b_n^{(\cdot, j)})(\mathbbm{x})\frac{\partial }{\partial \mathbbm{x}^{(j)}}\right)F_n.
\end{align*}
Finally, by setting $s = 0$, we have
\begin{align*}
    F_n(0, \mathbbm{x}; \lambda) & = \mathbb{L}_n(t, \mathbbm{x}; \lambda);
    \\
    \frac{\partial}{\partial s} F_n(0, \mathbbm{x}; \lambda) & = - \frac{\partial}{\partial t}\mathbb{L}_n(t, \mathbbm{x}; \lambda);
    \\
    \frac{\partial}{\partial \mathbbm{x}^{(j)}} F_n(0, \mathbbm{x}; \lambda) & = \frac{\partial}{\partial \mathbbm{x}^{(j)}}\mathbb{L}_n(t, \mathbbm{x}; \lambda);
    \\
    \frac{\partial^2}{\partial \mathbbm{x}^{(j_1)}\partial \mathbbm{x}^{(j_2)}} F_n(0, \mathbbm{x}; \lambda) & = \frac{\partial^2}{\partial \mathbbm{x}^{(j_1)}\partial \mathbbm{x}^{(j_2)}}\mathbb{L}_n(t, \mathbbm{x}; \lambda).
\end{align*}
Hence, the Equation \eqref{eq:general pde} must be satisfied by $\mathbb{L}_n$. 
\\
Next, we provide the proof for the second part of the theorem, which guarantees the solution to the PDE in Equation \eqref{eq:general pde} subject to the initial condition $\mathbb{L}_n(0, \mathbbm{x}; \lambda) = 1$ is indeed the characteristic function we stated in Equation \eqref{eq:char generalized-signature process}. Let $f_n(t, \mathbbm{x};\lambda)$ be a solution to the Equation \eqref{eq:general pde} with initial condition $f_n(0, \mathbbm{x};\lambda) = 1$. Fix $t>0$, define $F_n(s, \mathbb{X}^n_s; \lambda) = f_n(t-s, \mathbb{X}^n_s; \lambda)$ and consider the process
\begin{equation*}
    N_s = Z_s F_n(s, \mathbb{X}^n_s; \lambda),\quad 0<s<t,
\end{equation*}
where $Z_s = \exp(\mathfrak{i}M_\lambda(\mathbb{X}^n_s))$. Using It\^o's product rule we have
\begin{equation}
\label{eq:ito product_2}
    dN_s = dZ_s F_n(s, \mathbb{X}^n_s; \lambda) + Z_s dF_n(s, \mathbb{X}^n_s; \lambda) + dZ_s dF_n(s, \mathbb{X}^n_s; \lambda),
\end{equation}
By It\^o's formula, it is easy to obtain the SDE driven by $Z_s$ and $F_n(s, \mathbb{X}^n_s; \lambda)$, as in Equations \eqref{eq:SDE dZ} and $\eqref{eq:SDE Fn}$. Substituting into the Equation \eqref{eq:ito product_2}, we have
\begin{align*}
    dN_s & = Z_s D(s, \mathbb{X}^n_s;\lambda) ds + Z_s \Sigma(s, \mathbb{X}^n_s;\lambda) dW_s,
\end{align*}
where $D(s, \mathbb{X}^n_s;\lambda)$ and $\Sigma(s, \mathbb{X}^n_s;\lambda)$ are defined in Equations \eqref{eq: martingale drift} and \eqref{eq: martingale diffusion} respectively. By applying a change of variable $t' = t-s$ and using the definition of $f_n$ we have $D(s, \mathbbm{x};\lambda) = 0$. Hence, $N_s$ is a local martingale. The diffusion term is given by $Z_s\Sigma(s, \mathbb{X}^n_s;\lambda)$.
Since it is assumed that, $\sigma_n$ satisfies the linear growth condition and $\frac{\partial f_n}{\partial \mathbbm{x}}$ is continuous and satisfies the polynomial growth condition, we conclude that $\Sigma(s, \mathbb{X}^n_s;\lambda)$ is uniformly $L_2$-integrable (see Lemma \ref{lemma:regularity conditional expectation}). Then, by Funibi-Tonelli's theorem we have
\begin{equation*}
    \mathbb{E}\bigg[\int_0^s Z^2_{\tau} \Sigma^2(\tau, \mathbb{X}^n_{\tau} ;\lambda) d{\tau} \bigg] = \int_0^s \mathbb{E} [Z^2_{\tau} \Sigma^2(\tau, \mathbb{X}^n_{\tau} ;\lambda) ]d{\tau} < \infty,
\end{equation*}
hence $N_s$ is a martingale. Finally, evaluating $N$ at $s = 0$, we obtain
\begin{align*}
     N_0 &= Z_0 f_n(t, \mathbb{X}^n_0; \lambda) = \mathbb{E}[N_t|\mathcal{F}_0] = \mathbb{E}[\exp(\mathfrak{i} M_\lambda(\mathbb{X}^n_t) | \mathcal{F}_0];
    \\
    f_n(t, \mathbb{X}^n_0; \lambda) & = \mathbb{E}[\exp(\mathfrak{i} M_\lambda(\mathbb{X}^n_t-\mathbb{X}^n_0) | \mathbb{X}^n_0 =\mathbbm{x}] = \mathbb{L}_n(t, \mathbb{X}^n_0; \lambda),
\end{align*}
noting the equality $\mathbb{X}^n_0 = \mathbbm{x}$ completes the proof.
\end{proof}

\begin{remark}
We provide in the Appendix \ref{appendix: alternative proof of main theorem} a similar result to Theorem \ref{theorem: general char pde} assuming that the characteristic function $\mathbb{L}_n$ admits a Taylor expansion. By imposing the additional condition, we are able to establish the relationship between the $\mathcal{L}_n$ and the expected signature of $\mathbb{X}^n_{[0, t]}$ and we showed the relation is linear. The rest of the procedure relies on the characterization of the expected signature of a time-homogeneous It\^o diffusion via the parabolic PDE \cite{ni2012expected_signature}.

\end{remark}
\section{Conditional characteristic function of L\'evy area}
\label{sec: levy derivation}
In this section, we study a particular case of the linear transform on the signature of Brownian motions. Let $W = (W_t^{(1)},\dots , W_t^{(d)})_{t\in[0,T]}$ be a $d$-dimensional Brownian motion. Consider the process
\begin{equation}
\label{eq:almost_levy}
    L^{(i,j)}_t = \frac{1}{2}\bigg(\int_0^t W_s^{(i)} dW_s^{(j)} - \int_0^t W_s^{(j)} dW_s^{(i)}\bigg), \quad t\in [0, T],
\end{equation}
for any $1\leq i,j \leq d$. This process depends on the initial point of $W_0$ and links with L\'evy area. More specifically, when $W_0=0$, $L^{(i,j)}_t$ is the L\'evy area of $i$-th and $j$-th coordinates of a standard Brownian motion from time $0$ to $t$. We denote by $L_t = (L_t^{(i,j)})_{1\leq i<j\leq d}$ the $\frac{d(d-1)}{2}$-dimensional process. Denote by $skew(d)$ the space of $d$-dimensional skew-symmetric matrices and let $\Lambda \in skew(d)$, consider the sum process $L_t^{\Lambda} = \sum_{1\leq i< j \leq d} \Lambda_{i,j} L^{(i,j)}_t$. We explore the structure of $L_t^{\Lambda}$ by noting that it satisfies the SDE
\begin{equation}
\label{matrix_form_of_L}
    dL_t^{\Lambda} = 
\frac{1}{2}W_t^T \Lambda dW_t,\quad L^{\Lambda}_0 = 0.
\end{equation}
In particular, we are interested in the characteristic function of $L^{\Lambda}_t$ conditional on the starting point of Brownian motion, namely, 
\begin{align}
    \mathcal{L}(t, w; \cdot) & : skew(d) \to \mathbb{C};\notag
    \\
    (t,w; \Lambda) & \mapsto \mathbb{E}\bigg[\exp \bigg(\mathfrak{i} \frac{1}{2} \int_0^t W_s^T \Lambda dW_s\bigg) \bigg| W_0 = w\bigg] = \mathbb{E}[\exp(\mathfrak{i} L_t^{\Lambda}) | W_0 = w].\label{eq: conditional char levy}
\end{align}
In fact, $\mathcal{L}$ is a building block for deriving the analytical form of the joint characteristic function of Brownian motion coupled with its L\'evy area (See Section \ref{sec: Characteristic function of coupled Brownian motion and Levy area}, Lemma \ref{Lemma_link_L_and_Psi}). In contrast to \cite{helmes1983levy}, our proof does not rely on the change of measure via Girsanov's theorem and analytical continuation techniques used in complex analysis.  We instead discover a simplified PDE stated in Theorem \ref{theorem: general char pde} to this problem, which enables us to reduce the parabolic PDE into a system of Riccati equations for the case that all the eigenvalues of $\Lambda$ are non-zero.
\\
We establish in this section the relationship between $\mathcal{L}$ and the characteristic function of the generalized signature. To this end, we consider the generalized signature process of $W$ truncated at degree $2$:
\begin{equation*}
    \mathbb{X}^2_t = \bigg(1, \int_0^t \circ dW_s, \int_0^t \int_0^s (\circ dW_u) \otimes (\circ dW_s) + \rho^1(\mathbbm{x}) \otimes  \int_0^t \circ dW_s \bigg),\quad \mathbb{X}^2_0=\mathbbm{x}.
\end{equation*}
Let $\mathbb{L}_2(t, \mathbbm{x}; \lambda)$ be the characteristic function of $\mathbb{X}_2$, then we have the following lemma.

\begin{proposition}
For any $(t, w) \in \mathbb{R}^{+} \times \mathbb{R}^d$ and $\Lambda \in skew(d)$, it holds that
\begin{equation*}
        \mathcal{L}(t,w;\Lambda) = \mathbb{L}_2\left(t,(1, w); (0, 0, \frac{1}{2}\Lambda)\right).
    \end{equation*}
\end{proposition}
\begin{proof}
    It is proven by directly comparing the definition of $\mathcal{L}$ and $\mathbb{L}_2$.
\end{proof}
From now on, we stick with the notation $\mathcal{L}$ throughout this chapter for simplicity. 

In the following Lemma, we derive for any $\Lambda \in skew(d)$, a simplified PDE which $\mathcal{L}(t,w; \Lambda)$ satisfies. Moreover, the simplified PDE could be transformed into a system of Riccati equations and solved analytically. 

\begin{lemma}[PDE of the characteristic function $\mathcal{L}(t,w; \Lambda)$]
\label{PDE_of_conditional_Levy_area}
    Fix $\Lambda\in skew(d)$, then, $\mathcal{L}$ defined in Equation \eqref{eq: conditional char levy} is the solution to the following PDE
    \begin{equation}
    \label{main}
        \bigg(-\frac{\partial}{\partial t} + \frac{1}{2} \Delta \bigg) \mathcal{L}(t,w; \Lambda) + \frac{\mathfrak{i}}{2}\bigg(\sum_{j=1}^d (w^T \Lambda)_j \frac{\partial \mathcal{L}(t,w; \Lambda)}{\partial w_j} \bigg) - \frac{1}{8} w^T \Lambda \Lambda^T w \mathcal{L}(t,w; \Lambda) = 0,
    \end{equation}
    with initial condition $\mathcal{L}(0,w; \Lambda)=1$, where $\Delta$ is the Laplacian operator and $(t, w) \in \mathbb{R}^+\times \mathbb{R}^d$.
\end{lemma}
\begin{proof}
Let $\mathbb{X}^2$ be the generalized-signature process of $W$ with $\rho^1(\mathbb{X}^2_0) = W_0 = w$. Let $\mu_2$ and $\sigma_2$ be the drift and diffusion of $\mathbb{X}^2$. Let $\Lambda \in skew(d)$, consider $\Tilde{\lambda} = (0,0,\frac{1}{2}\Lambda)$ and $\mathbb{L}_2(t, (1, w); \Tilde{\lambda})$. It is clear that $L^{\Lambda}_t = M_{\Tilde{\lambda}}(\mathbb{X}^2_t)$. 
    Then, by comparing the SDE driven by $L^{\Lambda}_t$ and $M_{\Tilde{\lambda}}(\mathbb{X}^2_t)$, and using the fact that $M_{\Tilde{\lambda}}$ is linear we have
    \begin{equation*}
        M_{\Tilde{\lambda}} \circ \mu_2 = 0,\quad M_{\Tilde{\lambda}} \circ \sigma_2 = \frac{1}{2} w^T\Lambda.
    \end{equation*}
    Furthermore, by representing $\sigma_2$ as a matrix, we know that
    \begin{equation*}
    \sigma_2(w) = 
\begin{pmatrix}
    I_d \\
    \\
    B(w)
\end{pmatrix},\quad 
b_2(w) = \sigma_2 \sigma^T_2(w) =  \begin{pmatrix}
    I_d & B^T(w)\\
    \\
    B(w) & B\circ B^T(w)
\end{pmatrix},
\end{equation*}
for some matrix $B(w)$. Using the block nature of $\sigma_2$ we see that $b_2^{(\cdot, j)}$ is actually the $j$-th column of $\sigma$ for $1\leq j \leq d$. Therefore, 
\begin{align*}
    M_{\Tilde{\lambda}} \circ b_2^{(\cdot, j)}(w) & = \frac{1}{2} (w^T \Lambda)_j,\quad 1\leq j\leq d;
    \\
    M_{\Tilde{\lambda}}^{\otimes 2}\circ b_2 (w) & = M_{\Tilde{\lambda}}^{\otimes 2}\circ (\sigma_2 \sigma^T_2) = (M_{\Tilde{\lambda}} \circ \sigma_2) (M_{\Tilde{\lambda}} \circ \sigma_2)^T = \frac{1}{4} w^T\Lambda \Lambda^T w.
\end{align*}
Using Theorem \ref{theorem: general char pde} we know that $\mathbb{L}_2(t, (1, w); \Tilde{\lambda})$ satisfies the following PDE
\begin{align}
\label{eq:intermediate eq}
    \bigg(-\frac{\partial}{\partial t} + A_2 \bigg) \mathbb{L}_2(t, (1, w); \Tilde{\lambda}) & + \frac{\mathfrak{i}}{2}\bigg(\sum_{j=1}^d (w^T \Lambda)_j \frac{\partial \mathbb{L}_2(t, (1, w); \Tilde{\lambda})}{\partial w_j} \bigg) 
    \\
    & - \frac{1}{8} w^T \Lambda \Lambda^T w \mathbb{L}_2(t, (1, w); \Tilde{\lambda}) = 0,
\end{align}
with initial condition $\mathbb{L}_2(0, (1, w); \Tilde{\lambda})= 1$. Since $\mathbb{L}_2$ only depends on $w$ and $M_{\Tilde{\lambda}} \circ \mu_2 = 0$, the infinitesimal generator $A_2$ is the Laplacian. Thus Equation \eqref{eq:intermediate eq} becomes
\begin{align*}
    \bigg(-\frac{\partial}{\partial t} + \frac{1}{2} \Delta \bigg) \mathbb{L}_2(t, (1, w); \Tilde{\lambda}) & + \frac{\mathfrak{i}}{2}\bigg(\sum_{j=1}^d (w^T \Lambda)_j \frac{\partial \mathbb{L}_2(t, (1, w); \Tilde{\lambda})}{\partial w_j} \bigg) \\
    & - \frac{1}{8} w^T \Lambda \Lambda^T w \mathbb{L}_2(t, (1, w); \Tilde{\lambda}) = 0.
\end{align*}
Finally, we complete the proof by replacing $\mathbb{L}_2(t, (1, w); \Tilde{\lambda})$ with $\mathcal{L}(t,w; \Lambda)$.
\end{proof}

We derive the analytic formula for $\mathcal{L}(t, w; \Lambda)$ by exploiting the anti-symmetric property of matrix $\Lambda$. To do so, we first recall here a structural result on anti-symmetric matrices that will help us.
\begin{lemma}[Decomposition of anti-symmetric matrix]\label{lem_antisym_matrix_decomposition}
For any $d \times d$ anti-symmetric real-valued matrix $\Lambda$, let $(\pm \eta_{1}i,\dots, \pm \eta_{d_1}i)$ be the set of non-zero conjugate eigenvalue pairs of $\Lambda$ with $\eta_i > 0$ and $\eta_1\geq,\dots, \geq \eta_{d_1}$. Let $d_0 = d-2d_1$ be the algebraic multiplicity of the eigenvalue 0 of $\Lambda$ (if $\Lambda$ does not have zero eigenvalue, then $d_0=0$). Then there exists an orthogonal matrix $O$, such that the following decomposition of $\Lambda$ holds:
\begin{eqnarray*}
\Lambda = O^T \Sigma O,
\end{eqnarray*}
where $\Sigma$ is in the form that
\begin{eqnarray*}
\Sigma = \begin{pmatrix}
0 & -\eta_{1} & 0 &0& \dots & 0 &0 & 0 & \dots&0\\
\eta_{1} & 0 & 0 & 0 & \dots &  0&0& 0 & \dots&0 \\
0 & 0  & 0& -\eta_{2} & \dots & 0&0& 0 & \dots&0\\
0 &0& \eta_{2} &0 &\dots &0&0&0 & \vdots&0\\
\vdots & \vdots & \vdots & \vdots&\ddots&\vdots &\vdots&\vdots &\dots& \vdots\\
0 & 0 & \dots & 0&\dots &0& -\eta_{d_1} &0& \dots  & 0  \\
0 & 0 & \dots & 0&\dots &\eta_{d_1} & 0 &  0& \dots  & 0 \\
0 & 0 & \dots & 0 & 0 &0 & 0 &0& \dots  & 0 \\
0 & 0 & \dots & 0 & 0 &0 & 0 & \vdots & \ddots&\vdots\\
0 & 0 & \dots & 0 & 0 &0 & 0 & 0 & \dots&0\\
\end{pmatrix}:=\begin{pmatrix}
\Sigma_0,& \mathbf{0}_{d-d_0, d_0}\\
\mathbf{0}_{d_0, d-d_0}&\mathbf{0}_{d_0,d_0}
\end{pmatrix},
\end{eqnarray*}
where $\Sigma_0$ is the block diagonal matrix with all non-zero $\eta_i$ and $\mathbf{0}_{m,n}$ is the zero matrix of dimension $m\times n$..
\end{lemma}
\begin{proof}
    This is a classical result from spectral theory, check for example \cite{Greub1967}.
\end{proof}
\begin{lemma}[Analytic formula for $\mathcal{L}(t, w; \Lambda)$]\label{Lemma_formula_CharFunc_levy_area}
Let $d \geq 2$ be an even integer. Let $\Lambda$ be any skew-symmetric, invertible real-valued matrix of dimension $d \times d$. Then it holds that  
\begin{equation}
\label{formula_L}
    \mathcal{L}(t,w; \Lambda) = \bigg(\prod_{i=1}^{d'} \frac{1}{\cosh (\frac{\gamma_i}{2}t )} \bigg) \exp\bigg(\sum_{i=1}^{d'} -\frac{\eta_i}{4} ((Ow)^2_{2i-1} + (Ow)^2_{2i}) \tanh(\frac{\eta_i}{2}t)\bigg),
\end{equation}
where $d' = \frac{d}{2}$, $O,\ \eta $ are defined in Lemma \ref{lem_antisym_matrix_decomposition}.
\end{lemma}
\begin{proof}
Let $O$ be the orthogonal matrix obtained from $\Lambda$ by Lemma \ref{lem_antisym_matrix_decomposition}. If we choose the ansatz $\mathcal{L}(t,w;\Lambda) = f(t, y)$ where $y \coloneqq ((Ow)_1^2,\dots,(Ow)_d^2)$ we can use chain rule to obtain the following equalities
\begin{align*}
    \frac{\partial}{\partial t} \mathcal{L} & = \frac{\partial}{\partial t} f; 
    \\
    \frac{\partial}{\partial w_i} \mathcal{L} & = \sum_{j=1}^d  \frac{\partial f}{\partial y_j} \frac{d y_j}{d w_i} = \sum_{j=1}^d 2  \frac{\partial f}{\partial y_j} (Ow)_j (O)_{j,i};
    \\
    \frac{\partial^2}{\partial b_i^2} \mathcal{L}
    & = 
    \sum_{j=1}^d  \frac{\partial^2 f}{\partial y_j^2} \bigg( \frac{d y_j}{d w_i}\bigg)^2 + \sum_{j=1}^d  \frac{\partial f}{\partial y_j} \frac{d^2 y_j}{d w_i^2} 
    =
    4 \sum_{j=1}^d  \frac{\partial^2 f}{\partial y_j^2} (Ow)_j^2 (O_{j,i})^2 + 2 \sum_{j=1}^d  \frac{\partial f}{\partial y_j} (O_{j,i})^2.
\end{align*}
First, note that
\begin{equation}
\label{main1}
    \sum_{i=1}^d (w^T \Lambda)_i \frac{\partial \mathcal{L}(t,w;\Lambda)}{\partial w_i} = 2 \sum_{i=1}^d (w^T \Lambda)_i \sum_{j=1}^d   \frac{\partial f}{\partial y_j} (Ow)_j (O)_{j,i} = w^TO^T\Sigma \text{diag}\bigg(\frac{\partial f}{\partial y_i}\bigg)O w = 0,
\end{equation}
where $\text{diag}\bigg(\frac{\partial f}{\partial y_i}\bigg)$ is the diagonal matrix formed by the partial derivatives. In the penultimate equality, we express it in matrix form and by noting that $\Sigma \text{diag}\bigg(\frac{\partial f}{\partial y_i}\bigg)$ is anti-symmetric we have the last equality. Secondly, we have
\begin{equation}
\label{main2}
    w^T\Lambda\Lambda^Tw = \sum_{i=1}^{d'} \eta_i^2 (y_{2i-1} + y_{2i}),
\end{equation}
where $d' =  \frac{d}{2}$.
Before deriving the expression for the Laplacian of $\mathcal{L}$, we make the following ansatz: let $f(t,y) = g(t)\exp(\sum_i h_i(t) y_i))$, $i = 1,\dots, d$. The initial condition of $\mathcal{L}$ implies that $g(0) = 1,\ h_i(0)=0$ for all $i$ and
\begin{align}
\label{main3}
    \frac{1}{2\exp(\sum_i h_i(t) y_i)}\Delta \mathcal{L} & = 2 g(t) \sum_{i=1}^d \sum_{j=1}^d h_j^2(t) (O)_{j,i}^2  y_j +  g(t) \sum_{i=1}^d \sum_{j=1}^d h_j(t) (O)_{j,i}^2 \notag
    \\
    & = 2 g(t) \sum_{j=1}^d h_j^2(t)y_j + g(t) \sum_{j=1}^d h_j(t),
\end{align}
the last equality holds since $O$ is orthogonal. Finally, putting equations \eqref{main}, \eqref{main1}, \eqref{main2}, \eqref{main3} together we obtain
\begin{equation*}
    -g'(t) - g(t)\sum_{i=1}^d h'_i(t)y_i + 2 g(t) \sum_{i=1}^d h_i^2(t)y_i + g(t) \sum_{i=1}^d h_i(t) = \frac{1}{8} \sum_{i=1}^{d'} \eta_i^2 (y_{2i-1} + y_{2i}).
\end{equation*}
By rearranging in terms of $y_j$'s we establish the following system of ODEs
\begin{align*}
    & h'_i(t) = 2 h^2_i(t) - \frac{1}{8} \eta^2_{\lfloor i+1/2 \rfloor}, \quad i = 1,\dots,d;
    \\
    & g'(t) - g(t) \sum_{j=1}^d h_j(t) = 0.
\end{align*}
Since $\Lambda$ is invertible, $\eta_i \neq 0$ for all $i$, and the first equation becomes of Riccati type and the solution is given by 
\begin{equation*}
    h_i(t) = - \frac{\eta_{\lfloor i+1/2 \rfloor}}{4} \tanh (\frac{\eta_{\lfloor i+1/2 \rfloor}}{2}t), \quad \eta_{\lfloor i+1/2 \rfloor} \neq 0.
\end{equation*}
Finally, given $h_i(t)$, $g(t)$ has the form of
\begin{equation*}
    g(t) = g(0) \exp\left(\int_{0}^{t} \sum_{i = 1}^{d}h_i(s)ds\right) = \prod_{i=1}^{d} \underbrace{\exp\left(\int_0^{t}h_i(s)ds)\right)}_{C(t, \gamma_i)},
\end{equation*}
where
\begin{equation*}
    C(t, \gamma_i) = 
    \frac{1}{\sqrt{\cosh (\frac{\gamma_i}{2}t )}},
\end{equation*}
and $g(0) = 1$.
Therefore, the characteristic function of the sum process $L_t$ becomes
\begin{align*}
    \mathcal{L}(t,w; \Lambda) & = \bigg(\prod_{i=1}^{d} C(t, \gamma_i) \bigg) \exp\bigg(\sum_{i=1}^{d} h_i(t) (Ow)^2_{i}\bigg)
    \\
    & = \bigg(\prod_{i=1}^{d'} \frac{1}{\cosh (\frac{\gamma_i}{2}t )} \bigg) \exp\bigg(\sum_{i=1}^{d'} -\frac{\eta_i}{4} ((Ow)^2_{2i-1} + (Ow)^2_{2i}) \tanh(\frac{\eta_i}{2}t)\bigg).
\end{align*}
\end{proof}

\section{Characteristic function of Brownian motion coupled with the L\'evy area}\label{sec: Characteristic function of coupled Brownian motion and Levy area}
We recall the definition of the joint characteristic function in Theorem \ref{main_theorem}
\begin{equation}
\label{joint_characteristic_function}
    \Psi_{W}(t, \mu, \Lambda) = \mathbb{E} \bigg[ \exp \left(\mathfrak{i} \sum_{i=1}^d \mu_i W_t^{(i)}+ i L^{\Lambda}_t \right)\bigg \vert W_0 = 0 \bigg].
\end{equation}

We establish the connection between the characteristic function of L\'evy area conditioning on that $W_0 = w$ and the characteristic function of Brownian motion coupled with the L\'evy area conditioning on that $W_0 = 0$ by the translation invariance of Brownian motion in the following lemma.
\begin{lemma}[Connection between $\mathcal{L}$ and $\Psi$]\label{Lemma_link_L_and_Psi}
Let $\mathcal{L}$ and $\Psi$ defined the Equations \eqref{eq: conditional char levy} and \eqref{joint_characteristic_function} respectively, then, for any $(t, w)\in \mathbb{R}\times \mathbb{R}^d$ and $\Lambda\in skew(d)$,
\begin{eqnarray}
\mathcal{L}(t, w; \Lambda) = \Psi_{\Tilde{W}}(t, \frac{1}{2}w^T\Lambda, \Lambda),
\end{eqnarray}
where $\Tilde{W}$ is a Brownian motion with $\Tilde{W}_0 = 0$.
\end{lemma}
\begin{proof}
    Let $1\leq i<j\leq d$, consider $L^{(I,j)}_t$ defined in Equation \eqref{eq:almost_levy}. Assume that $W_0 = w$ and $w\neq 0$, then $W_t$ can be further expressed as $W_t = w + \Tilde{W}_t$ where $\Tilde{W}$ is a standard Brownian motion, i.e. $\Tilde{W}_0 = 0$. Therefore,
    \begin{equation*}
    L^{(i,j)}_t = \frac{1}{2} \bigg(\int_0^t W_s^{(i)} dW_s^{(j)} - \int_0^t W_s^{(j)} dW_s^{(i)} \bigg) = \frac{1}{2}w^{(i)} \Tilde{W}_t^{(j)} - \frac{1}{2} w^{(j)} \Tilde{W}_t^{(i)} + \Tilde{L}_t^{(i,j)},
\end{equation*}
where 
\begin{equation*}
    \Tilde{L}_t^{(i,j)} = \frac{1}{2}\int_0^t \Tilde{W}_s^{(i)} d\Tilde{W}_s^{(j)} - \int_0^t \Tilde{W}_s^{(j)} d\Tilde{W}_s^{(i)},
\end{equation*}
is the L\'evy area of a standard Brownian motion. Hence, we can transform $\mathcal{L}(t,w;\Lambda)$ into the joint characteristic function of L\'evy area and standard Brownian motion
\begin{align*}
    \mathcal{L}(t,w; \Lambda) & = \mathbb{E}[\exp (\mathfrak{i} L^{\Lambda}_t) \bigg| W_0=w]
    \\
    & = \mathbb{E} \bigg( \exp \mathfrak{i}(\frac{1}{2}w^T \Lambda \Tilde{W}_t  + \Tilde{L}^{\Lambda}_t) \bigg \vert \Tilde{W}_0 = 0\bigg) 
    \\
    & =\Psi_{\Tilde{W}}(t, \frac{1}{2}w^T\Lambda, \Lambda).
\end{align*}
\end{proof}
In Section \ref{sec: levy derivation} we derived the analytical form for $\mathcal{L}$ when $\Lambda$ is non-degenerate, i.e. it does not possess any zero eigenvalue. Lemma \ref{Lemma_link_L_and_Psi} is therefore useful to transform the problem of computing $\Psi_W$ to $\mathcal{L}$, as $\Psi_W(t, \mu, \Lambda)$ is the same as $\mathcal{L}(t, 2(\Lambda^T)^{-1}\mu;\Lambda)$ provided that $\Lambda$ is invertible. However, additional care needs to be taken for the general case when $\Lambda$ possesses zero eigenvalues. In fact, when $d$ is odd, $\Lambda$ has at least one zero eigenvalue, therefore not invertible. We alleviate this degenerate issue by splitting the sum process $L_t^{\Lambda}$ into (1) the non-degenerate part composed with a lower dimension Brownian Brownian motion and the corresponding L\'evy area and (2) the independent Brownian motion component. This decomposition takes advantage of the rotational invariance of the sum process $L_t^{\Lambda}$ with respect to Brownian motions and structural decomposition of anti-symmetric matrix $\Lambda$. By doing so, we are able to derive the joint characteristic function $\Psi$ as the product of the characteristic function of those two parts, which are both already worked out. To our best knowledge, this degeneration issue has not been discussed in past literature.

\subsection{The case for non-degenerate $\Lambda$}


We begin with a simplified case where $\Lambda$ is invertible, note that this condition is satisfied only if $d$ is even. The joint characteristic function is given by the following theorem.
\begin{thm}
\label{theorem_full_case}
Let $W = (W_t^{(1)},\dots ,W_t^{(d)})_{t\in [0, T]}$ be a $d$-dimensional Brownian motion with $d$ even. Let $\mu\in \mathbb{R}^d$ and let $\Lambda\in \mathbb{R}^{d\times d}$ be a skew-symmetric and invertible matrix. Then the joint characteristic function $\Psi_{W}(t, \mu, \Lambda)$ defined in Equation \eqref{joint_characteristic_function} is given by
\begin{equation}\label{eqn_C}
    \Psi_{W}(t, \mu, \Lambda) = \bigg( \prod_{i=1}^{d'} \frac{1}{\cosh(\frac{\eta_i}{2}t)} \bigg) \exp \bigg(\sum_{i=1}^{d'} - \frac{1}{\eta_i} ((O\mu)_{2i-1}^2+(O\mu)_{2i}^2)\tanh(\frac{\eta_i}{2}t)\bigg), 
\end{equation}
where $d' = \frac{d}{2}$, $O,\ \eta $ are defined in Lemma \ref{lem_antisym_matrix_decomposition}.
\end{thm}
\begin{proof}
Let $d'=d/2$. By Lemma \ref{Lemma_formula_CharFunc_levy_area}, we derived the characteristic function of $L_t^{\Lambda}$, conditioned on an arbitrage starting point $W_0 = w$ of the Brownian motion. If $\Lambda$ is invertible then $w^T\Lambda$ spans the space of $\mathbb{R}^d$, hence $\mathcal{L}(t,w;\Lambda)$ can fully characterize $\Psi_{W}(t, \mu, \Lambda)$. Mathematically, for any $\mu \in \mathbb{R}^d$, we have
\begin{equation*}
    \Psi_{W}(t, \mu, \Lambda) = \mathcal{L}(t, 2(\Lambda^T)^{-1}\mu;\Lambda),
\end{equation*}
substituting $2(\Lambda^T)^{-1}\mu$ into equation \eqref{formula_L} and by noting that $(\Lambda^T)^{-1} = O^T (\Sigma^T)^{-1}O$ with
\begin{equation*}
    (\Sigma^T)^{-1} = 
\begin{pmatrix}
0 & -\frac{1}{\eta_{1}} & 0 &0& \dots & 0 &0 \\
\frac{1}{\eta_{1}} & 0 & 0 & 0 & \dots &  0&0 \\
0 & 0  & 0& -\frac{1}{\eta_{2}} & \dots & 0&0\\
0 &0& \frac{1}{\eta_{2}} &0 &\dots &0&0\\
\vdots & \vdots & \vdots & \vdots&\ddots&\vdots &\vdots\\
0 & 0 & \dots & 0&\dots &0& -\frac{1}{\eta_{\frac{d}{2}}}  \\
0 & 0 & \dots & 0&\dots & \frac{1}{\eta_{\frac{d}{2}}}& 0 &  
\end{pmatrix},
\end{equation*}
hence $(Ow)^{2}_{2i-1} + (Ow)^{2}_{2i} = \frac{4}{\eta^2_i} ((O\mu)^{2}_{2i-1} + (O\mu)^{2}_{2i})$ for $1\leq i \leq d'$. Finally
\begin{equation*}
    \Psi_{W}(t, \mu, \Lambda) = \bigg( \prod_{i=1}^{d'} \frac{1}{\cosh(\frac{\eta_i}{2}t)} \bigg) \exp \bigg(\sum_{i=1}^{d'} - \frac{1}{\eta_i} ((O\mu)_{2i-1}^2+(O\mu)_{2i}^2)\tanh(\frac{\eta_i}{2}t)\bigg). 
\end{equation*}
\end{proof}
In the following, we use $\Psi^{\text{Non-deg}}_{W}(t, \mu, \Lambda)$ to emphasize that it is the joint characteristic function $\Psi_{W}(t, \mu, \Lambda)$ with the non-degenerate $\Lambda$.
\subsection{A general case for anti-symmetric matrix $\Lambda$}
We recall the general statement for $\Psi_W(t,\mu, \Lambda)$, the joint characteristic function of standard Brownian motion and its L\'evy area that includes the case where the parameter matrix $\Lambda$ is degenerate.

\levychar

 
The proof consists of three steps. First, we consider the joint process and split the $d$-dimensional Brownian motion into two groups (after some rotation), namely, $\hat{W}$ and $\hat{W}^{\indep}$ such that $\hat{W}^{\indep}_t$ is independent of $\hat{W}$ and the L\'evy process. Secondly, we use Theorem \ref{theorem_full_case} to obtain the joint characteristic function of $\hat{W}$ and the corresponding L\'evy area. Finally, using the independence fact between $\hat{W}$ and $\hat{W}^{\indep}$ we derive the full joint characteristic function conditional on zero starting point.
\\
The following lemma shows the transformation of $\Psi_{W}$ by any rotation on $W$.
\begin{lemma}
For any orthogonal matrix $M$ of dimension $d \times d$,
\begin{eqnarray*}
\Psi_{W}(t,M\mu, M\Lambda M^T) =  \Psi_{M^TW}(t,\mu, \Lambda).
\end{eqnarray*}
\end{lemma}
\begin{proof}
    Note that $\sum_{i=1}^d (M\mu)_i W_t^{(i)} = \langle M\mu, W_t\rangle = \mu^TM^TW_t = \langle \mu, M^TW_t\rangle$. Hence
    \begin{align*}
        \Psi_{M^TW}(t, \mu, \Lambda) & = \mathbb{E} \bigg[ \exp \left(\mathfrak{i} \langle \mu, M^TW_t\rangle+ i (M^TW)^T \Lambda d(M^TW_t) \right)\bigg|W_0 = 0 \bigg] 
        \\
         & =  \mathbb{E} \bigg[ \exp \left(\mathfrak{i} \langle M\mu, W_t\rangle+ i W^T_t M \Lambda M^TdW_t \right)\bigg|W_0 = 0 \bigg]
        \\
         & = \Psi_{W}(t,M\mu, M\Lambda M^T).
    \end{align*}
\end{proof}
This lemma will help us on choosing $\hat{W}$ and $\hat{W}^{\indep}$. Now, we provide the proof for the general case, where $\Lambda$ does have zero eigenvalue:
\begin{proof}{(\textit{Theorem \ref{main_theorem}})}
    Let $W_t = (W_t^{(1)},\dots W_t^{(d)})$ be a $d$-dimensional Brownian motion and let $\Lambda$ be the matrix associated to the sum process $L_t^{\Lambda}$ of $W_t$. By Lemma \ref{lem_antisym_matrix_decomposition} we have that $\Lambda = O^T\Sigma O$. Based on $\Sigma$, we know the number of zero eigenvalues $d_0$. If $d_0 = 0$, $\Lambda$ is invertible and the problem is solved by Theorem \ref{theorem_full_case}. If $d_0>0$, then let $(O\mu)_{1:d-d_0} = \xi$ and $(O\mu)_{d-d_0:d} = \xi_0$ and define $\hat{W}_t = ((OW_t)^{(1)},\dots,(OW_t)^{(d-d_0)})$, $\hat{W}_t^{\indep} = ((OW_t)^{(d-d_0+1)},\dots,(OW_t)^{(d)})$. $\hat{W}_t^{\indep}$ is independent of $\hat{W}_t$ hence independent of $L_t^{\Lambda}$. We have that
\begin{eqnarray*}
\Psi_{W}(t, \mu, \Lambda) = \Psi_{OW}(t, O\mu, \Sigma) &=& \Psi_{OW}\left(t, (\xi, \xi_0)^{T}, \begin{pmatrix}
\Sigma_0,& \mathbf{0}_{d-d_0, d_0}\\
\mathbf{0}_{d_0, d-d_0}&\mathbf{0}_{d_0,d_0}
\end{pmatrix} \right)
\\
&=&\Psi_{\hat{W}}(t, \xi, \Sigma_0) \mathbb{E}[\exp(i \langle \xi_0, \hat{W}_t^{\indep}\rangle)\vert \hat{W}_0^{\indep} = 0]\\
&=&\Psi_{\hat{W}}(t, \xi, \Sigma_0) \exp\left(-\frac{1}{2}t\sum_{i = d-d_0}^d (\xi_0)^2_{i}\right).
\end{eqnarray*}
Hence we can reduce the problem into a degenerate space with dimension $d-d_0$ and $\Sigma_0 \in \mathbb{R}^{d-d_0 \times d-d_0}$ being invertible. By Theorem \ref{theorem_full_case}, we know that 
\begin{align*}
    \Psi_{\hat{W}}(t, \xi_1, \Sigma_0) & = \Psi^{\text{Non-deg}}_{\hat{W}}(t, \xi_1, \Sigma_0) \\
    & = \bigg( \prod_{i=1}^{\frac{d-d_0}{2}} \frac{1}{\cosh(\frac{\eta_i}{2}t)} \bigg) \exp \bigg(\sum_{i=1}^{\frac{d-d_0}{2}} - \frac{1}{\eta_i} (\xi_{2i}^2+\xi_{2i+1}^2)\tanh(\frac{\eta_i}{2}t)\bigg),
\end{align*}
where $\eta_i$'s depends on $\Sigma_0$ in the way stated in Lemma \ref{lem_antisym_matrix_decomposition}. Therefore,
\begin{align*}
    \Psi_{W}(t, \mu, \Lambda) = & \bigg( \prod_{i=1}^{d_1} \frac{1}{\cosh(\frac{\eta_i}{2}t)} \bigg) 
    \\
    & \exp \bigg( \big[ \sum_{i=1}^{d_1} - \frac{1}{\eta_i} ((O\mu)_{2i}^2+(O\mu)_{2i+1}^2)\tanh(\frac{\eta_i}{2}t) \big] - \frac{1}{2}t \sum_{i=1}^{d_0} (O\mu)^2_{2d_1+i}\bigg),
\end{align*}
with $d_0$, $d_1$ defined in Lemma \ref{lem_antisym_matrix_decomposition}.
\end{proof}
\section{Conclusion and future work}


In this paper, we provided a PDE approach to derive the characteristic function of the signature of a time-homogeneous It\^o diffusion up to a fixed time. As an important intermediate step, we generalized the signature process and established the Feynman-Kac type theorem for the characteristic function of the generalized signature process, following the martingale approach. Furthermore, as an application, we presented an alternative and novel proof of the well-established result concerning the characteristic function of $d$-dimensional Brownian motion coupled with its associated L\'evy area

As an extension of our general methodology, it is of great interest to extend our work and establish the PDE theorem for the path characteristic function (PCF) \cite{lou2023pcfgan, Ilya2013char} of an It\^o diffusion process. The PCF, built on top of the unitary matrix of arbitrary order, can determine the law on the unparameterized path (signature) space. It is noteworthy that the characteristic function of the signature process, discussed in our paper, can be regarded as a special case of the PCF when the matrix order is 1. However, the non-commutativity of matrix multiplication with order $>2$ imposes substantial difficulty on the extension work. 



Moreover, focusing on the coupled Brownian motion and L\'evy area, it would be interesting to investigate the conditional law of the joint process given its increment and the mid-point. Hence, a concrete open question would be how to derive its characteristic function via the PDE approach. It may shed some light on novel sampling algorithms as an alternative to computationally heavy Monte Carlo methods.  This study is instrumental for adaptive numerical schemes of SDE solvers as well as the sequential neural networks motivated by SDEs, such as Logsig-RNN\cite{Li2020} and Neural SDEs\cite{Kidger2021}.



\appendix
 \section{$L_2$-integrability of $\Sigma$}
\label{appendix: l2 integrability}
We first cite here a well-known result that established the bounds for moments of a time-homogeneous It\^o diffusion process under mild regularity conditions.

\begin{lemma}[Krylov \cite{krylov1980cdp}, Corollary 12]
\label{lemma:krylov}
    Let $X = (X_t)_{t \in [0, T]}$ be a $E$-valued time-homogeneous It\^o diffusion process with drift $\mu$, diffusion $\sigma$ and the starting point $X_0 = x$, assume that $\mu$ and $\sigma$ satisfy Condition \ref{conditionVectorField}.
    \\
    Then for every suffciently small $T>0$, there exists a constant $N$ depending on $p$, such that for all $p>0$ and $t\in [0,T]$,
    \begin{align*}
        \mathbb{E}^x [\sup_{t\in[0,T]} \vert X_t - x \vert^p] & \leq NT^{\frac{q}{2}}e^{NT}(1+\vert x\vert)^p;
        \\
        \mathbb{E}^x [\sup_{t\in[0,T]} \vert X_t \vert^p] & \leq N e^{NT}(1+\vert x\vert)^p.
    \end{align*}
\end{lemma}

\begin{lemma}
\label{lemma:regularity conditional expectation}
    $Let, E, B$ be two Banach spaces. Let $X = (X_t)_{t \in [0, T]}$ be a $E$-valued time-homogeneous It\^o diffusion process with drift $\mu$, diffusion $\sigma$ and the starting point $X_0 = x$, assume that $\mu$ and $\sigma$ satisfy Condition \ref{conditionVectorField}. Let $f: \mathbb{R}^+ \times E \to B$ be a function that satisfies polynomial condition in Definition \ref{definition_linear_growth}. Then, for every fixed $T>0$, $f(t, X_t)$ is $L_p$-integrable, namely,
    \begin{equation*}
        \sup_{t\in[0,T]} \mathbb{E} [\vert f(t, X_t) \vert^p] < \infty.
    \end{equation*}
\end{lemma}
\begin{proof}
    By the polynomial growth of $f$, we have 
    \begin{equation*}
        \sup_{t\in[0,T]}\vert f(t,x)\vert \leq K_f(1+\vert x\vert^{r_f}),
    \end{equation*}
    for some constant $K_f$ and $r_f$. Then, for any $p>1$
\begin{align*}
    (1+\vert x\vert^{r_f})^p \leq K' (1+\vert x\vert^{2r_f p}),
\end{align*}
    for some constant $K'$ that depends on $r_f$ and $p$. Therefore, we have
    \begin{align*}
        \sup_{t\in[0,T]} \mathbb{E}^x [\vert f(t, X_t) \vert^p] \leq K_f^p \sup_{t\in[0,T]} \mathbb{E}^x [(1+\vert X_t\vert^{r_f})^p]\leq K  (1+\mathbb{E}^x [\sup_{t\in[0,T]}\vert X_t\vert^{2r_fp}]),
    \end{align*}
where $K=K_f^p K'$. Now using Lemma \ref{lemma:krylov}, we know that the RHS is bounded by $N e^{NT}(1+\vert x\vert)^{2r_fp}$ for some constant $N$, hence we conclude that
\begin{equation*}
    \sup_{t\in[0,T]} \mathbb{E}^x [\vert f(t, X_t) \vert^p] < \infty.
\end{equation*}
Finally,
\begin{equation*}
    \sup_{t\in[0,T]}\mathbb{E}[\vert f(t, X_t) \vert^p] = \sup_{t\in[0,T]} \mathbb{E}[\mathbb{E}^x[\vert f(t, X_t) \vert^p]] \leq  \mathbb{E}[\sup_{t\in[0,T]}\mathbb{E}^x[\vert f(t, X_t) \vert^p]] < \infty.
\end{equation*}
\end{proof}

\section{Alternative approach to PDE theorem}
\subsection{PDE characterization via Taylor approach}
\label{appendix: alternative proof of main theorem}
In this subsection, we prove the PDE characterization of $\mathbb{L}_n$ by establishing the relationship between the characteristic function and the expected signature of a diffusion process. The connection is obtained via Taylor's expansion of $\mathbb{L}_n$ provided the series converges. We introduce now the regularity condition that $\mathbb{L}_n$ such that the Taylor series is well-defined.
\begin{definition}[Radius of convergence]
Fix $n\in \mathbb{N}$, let $(\mathbb{X}^n_t)_{t\in[0,T]}$ be the process defined in Definition \ref{def:sig-aug process}. Let $\mathbb{L}_n(t, \mathbbm{x};\lambda)$ be defined as in Definition \ref{def:char generalized-signature process}. We call a scalar $\mathcal{R} \in \mathbb{R}^{+}$ the radius of convergence (ROC) of $\mathbb{L}_{n}$ if  
\begin{eqnarray}
    \mathcal{R}:= \sup_{\lambda \in \mathbb{R}^{D_n} \& \mathcal{A}}|\lambda|, 
\end{eqnarray}
where $D_n$ is the dimension of $T^n(E)$ and $\mathcal{A}$ represents the set of $\lambda$ such that for all $(t, \mathbbm{x})$
\begin{eqnarray*}
\mathbb{L}_{n}(t, \mathbbm{x};\lambda) = \sum_{k \geq 0} i^k M_\lambda^{\otimes k}(\Phi_m(t,\mathbbm{x})).
\end{eqnarray*}
holds.
\end{definition}

\begin{theorem}\label{theorem enhanced pde}
    Fix $n\in \mathbb{N}$, let $(\mathbb{X}^n_t)_{t\in[0,T]}$ be the process defined in Definition \ref{def:sig-aug process} with drift $\mu_n$ and diffusion $\sigma_n$.
    Denote by $D_n$ the dimension of $\mathbb{X}^n$. Let $\Phi(t,\mathbbm{x})$ be the expected signature of $\mathbb{X}_{[0,t]}$ with $\mathbb{X}_0 = \mathbbm{x}$. Assume 
    \begin{itemize}
        \item $\mu_n,\sigma_n$ satisfy the Condition \ref{conditionVectorField}.
        \item $\Phi(t,\mathbbm{x})$ is differentiable with respect to $t$ and twice differentiable with respect to $\mathbbm{x}$.
    \end{itemize}
    Let $M: T^n(E)\to \mathbb{R}$ be a linear functional. Then, for ant $m>2$, the following recursive PDE holds
\begin{align}\label{eq:recursive pde linear process}
\left(-\frac{\partial }{\partial t}+ A\right)M^{\otimes m}\circ\Phi_m(t,\mathbbm{x})  = &
- (M\circ \mu_n)(\mathbbm{x}) (M^{\otimes m-1}\circ\Phi_{m-1})(t,\mathbbm{x})
\notag
\\
& -\sum_{j=1}^{D_n}(M\circ b_n^{(\cdot, j)})(\mathbbm{x})\frac{\partial }{\partial \mathbbm{x}^{(j)}} (M^{\otimes m-1}\circ\Phi_{m-1})(t,\mathbbm{x})
\notag
\\
& - (M^{\otimes 2}\circ b_n)(\mathbbm{x}) (M^{\otimes m-2}\circ\Phi_{m-2})(t,\mathbbm{x}),
\end{align}
where $\mathbbm{x}^{(j)}$ denotes the $j$-th coordinate of $\mathbbm{x}$,  $b_n = \sigma_n \sigma_n^T$, $b_n^{(\cdot, j)}$ denotes the $j$-th columns of $b_n$ and
\begin{equation*}
     A_nf(\mathbbm{x}_0) = \sum_{i=1}^{D_{n-1}} \mu_n^{(i)}(\mathbbm{x}_0)\frac{\partial f}{\partial \mathbbm{x}^{(i)}}\Bigg \vert_{\substack{\mathbbm{x}=\mathbbm{x}_0}} + \sum_{j_1=1}^{D_{n-1}} \sum_{j_2=1}^{D_{n-1}} \frac{1}{2}b_n^{(j_{1}, j_{2})}(\mathbbm{x}_0) \frac{\partial^2 f}{\partial \mathbbm{x}^{(j_1)}\partial \mathbbm{x}^{(j_2)}}\Bigg \vert_{\substack{\mathbbm{x}=\mathbbm{x}_0}}.
\end{equation*}
Furthermore,
\begin{equation}
\label{eq:recursive pde linear process 2}
    \left(-\frac{\partial }{\partial t}+ A\right)M\circ\Phi_1(t,\mathbbm{x}) = -M \circ \mu_n(\mathbbm{x}).
\end{equation}
Moreover, for every interger $m\geq 1$, $M^{\otimes m}\circ\Phi_m(t,\mathbbm{x})$ satisfies
\begin{equation*}
    M^{\otimes m}\circ\Phi_m(0,\mathbbm{x}) = 0,
\end{equation*}
and
\begin{equation*}
    \Phi_0(t,\mathbbm{x}) = 1,\quad \forall t\in \R, \mathbbm{x} \in E\oplus T^{n}(E).
\end{equation*}
\end{theorem}

\begin{proof}
    By Lemma Corollary \ref{corollary_pde} we have that $\Phi(t,\mathbbm{x})$ satisfies the recursive PDE
\begin{align}
\left(-\frac{\partial }{\partial t}+ A\right)\Phi_m(t,\mathbbm{x})  = &
- \left(\sum_{j = 1}^{D_n} \mu_n^{(j)}(\mathbbm{x})e_{j}\right) \otimes \Phi_{m-1}(t,\mathbbm{x}) \notag
\\
& -\sum_{j=1}^{D_n}\left(\sum_{j_{1}=1}^{D_n}b_n^{(j_{1}, j)}(\mathbbm{x})e_{j_{1}}\right) \otimes \frac{\partial \Phi_{m-1}(t,\mathbbm{x})}{\partial \mathbbm{x}^{(j)}} \notag
\\
& - \left(\frac{1}{2}\sum_{j_{1}=1}^{D_n}\sum_{j_{2}=1}^{D_n}b_n^{(j_{1}, j_{2})}(\mathbbm{x})e_{j_{1}}\otimes e_{j_{2}}\right) \otimes \Phi_{m-2}(t,\mathbbm{x}), \label{eq:intermediate_pde}
\end{align}
for any $m>2$. Additionally, $\forall t\in \R, \mathbbm{x} \in T^{n}(E)$ we have
\begin{align}
    \left(-\frac{\partial }{\partial t}+ A\right)\Phi_1(t,\mathbbm{x}) & = - \left(\sum_{j = 1}^{D_n} \mu_n^{(j)}(\mathbbm{x})e_{j}\right) \label{eq:recursive1};
    \\
    \Phi_m(0,\mathbbm{x}) & = 0,\quad m\geq 1;\label{eq:recursive2}
    \\
    \Phi_0(t,\mathbbm{x}) & = 1.\label{eq:recursive3}
\end{align}
Note that
\begin{equation*}
    M^{\otimes m}(x \otimes y ) = M^{\otimes n}(x)  M^{\otimes m-n}(y),
\end{equation*}
for any $x\in T^{n}(E)$, $y\in T^{m-n}(E)$ for all $n\in \{0,\dots, m\}$. We apply corresponding $M^{\otimes m}$ on both sides of Equations \eqref{eq:intermediate_pde}, \eqref{eq:recursive1}, \eqref{eq:recursive2}, and use the linearity of $M$ we obtain the required result.
\end{proof}

Note that the PDE stated in Theorem \ref{theorem enhanced pde} can be further simplified by noting that $\Phi(t,\mathbbm{x})$, the expected signature of $\mathbb{X}^n_{[0, t]}$ does not depend on the last degree of the starting point $\mathbb{X}^n_0$.  
\begin{lemma}
\label{lemma_invariance_expected_signature}
    Let $(\mathbb{X}^n_t)_{t\in[0,T]}$ be the process defined in Definition \ref{def:sig-aug process}. For any $\mathbbm{x}=(1,\mathbbm{x}_1,\dots, \mathbbm{x}_n)\in T^n(E)$, the expected signature of $\mathbb{X}^n_{[0,t]}$ conditional on $\mathbb{X}^n_0 = (1,\mathbbm{x}_1,\dots, \mathbbm{x}_n)$, denoted by $\Phi(t,\mathbbm{x})$, is independent of $\mathbbm{x}_n$, namely,
\begin{align*}
    \Phi(t, (1, \mathbbm{x}_1, \dots, \mathbbm{x}_n)) & = \mathbb{E}[S(\mathbb{X}^n_{[0, t]}) \vert \mathbb{X}^n_0 = (1, \mathbbm{x}_1, \dots, \mathbbm{x}_n)]
    \\
    & = \mathbb{E}[S(\mathbb{X}^n_{[0, t]}) \vert \mathbb{X}^{n-1}_0 = (1, \mathbbm{x}_1,\dots,\mathbbm{x}_{n-1})]. 
\end{align*}
where $\text{Sig}$ denotes the signature transform.
\end{lemma}
\begin{proof}
    Let $\Tilde{\mathbbm{x}}_n \in (\mathbb{R}^d)^{\otimes n}$, $\Tilde{\mathbb{X}}^n_t = \mathbb{X}^n_t + (0,\dots,\Tilde{\mathbbm{x}}_n)$. Since $d\Tilde{\mathbb{X}}^n_t = d\mathbb{X}^n_t$ we have that $S(\Tilde{\mathbb{X}}^n_{[0,t]}) = S(\mathbb{X}^n_{[0,t]})$. Also, from Lemma \ref{lemma:enhanced diffusion} we know that $\mathbb{X}^n_t$ only depends on $\mathbbm{x}_1,\dots \mathbbm{x}_{n-1}$, therefore the expected signature will also agree, i.e. $\mathbb{E}[S(\Tilde{\mathbb{X}}^n_{[0,t]})] = \mathbb{E}[S(\mathbb{X}^n_{[0,t]})]$.
\end{proof}
By combining Theorem \ref{theorem enhanced pde} and Lemma \ref{lemma_invariance_expected_signature} and establishing the linkage between the characteristic function and the expected signature of an It\^o diffusion, we present here the main result. The PDE representation of the conditional characteristic function of the truncated signature of an It\^o diffusion via the Taylor series approach.

\begin{theorem}
\label{theorem: general char pde with radius of convergence}
Fix $n\in \mathbb{N}$, let $(\mathbb{X}^n_t)_{t\in[0,T]}$ be the process defined in Definition \ref{def:sig-aug process} with drift $\mu_n$ and diffusion $\sigma_n$. Denote by $\Phi$ the expected signature of $\mathbb{X}^n_{[0,t]}$ conditioned on the starting point $\mathbb{X}^n_0 = \mathbbm{x}$ and assume that $\mathbb{X}^n, \Phi(t,\mathbbm{x})$ satisfy conditions in Theorem $\ref{theorem enhanced pde}$. Denote by $D_n$ the dimension of $\mathbb{X}^n$. 
\\
Let $\lambda\in T^n(E)$ and consider the characteristic function $\mathbb{L}_n(t,\mathbbm{x}; \lambda)$ defined in Definition \ref{def:char generalized-signature process}. Finally, assume that, 
$|\lambda|\leq \mathcal{R}$ where $\mathcal{R}$ is the ROC of $\mathbb{L}_n$. Then, $\mathbb{L}_n$ is the solution to the following PDE: 
\begin{eqnarray}
\label{eq:general pde radius of convergence}
\left(-\frac{\partial }{\partial t}+ A_n\right)\mathbb{L}_n(t,\mathbbm{x}; \lambda) +i\left((M_{\lambda}\circ \mu_n)(\mathbbm{x})+\sum_{j=1}^{D_{n-1}}(M_{\lambda}\circ b_n^{(\cdot, j)})(\mathbbm{x})\frac{\partial }{\partial \mathbbm{x}^{(j)}}\right)\mathbb{L}_n(t,\mathbbm{x}; \lambda)\notag
\\
-\frac{1}{2}(M_{\lambda}^{\otimes 2}\circ b_n)(\mathbbm{x})\mathbb{L}_n(t,\mathbbm{x}; \lambda)=0, 
\end{eqnarray}
with initial condition $\mathbb{L}_n(0, \mathbbm{x}; \lambda) = 1$ for all $(t, \mathbbm{x})\in \R\times T^{n-1}(E)$, $b_n = \sigma_n \sigma_n^T$, $b_n^{(\cdot, j)}$ denotes the $j$-th column of $b_n$ and
\begin{equation*}
     A_nf(\mathbbm{x}_0) = \sum_{i=1}^{D_{n-1}} \mu_n^{(i)}(\mathbbm{x}_0)\frac{\partial f}{\partial \mathbbm{x}^{(i)}}\Bigg \vert_{\substack{\mathbbm{x}=\mathbbm{x}_0}} + \sum_{j_1=1}^{D_{n-1}} \sum_{j_2=1}^{D_{n-1}} \frac{1}{2}b_n^{(j_{1}, j_{2})}(\mathbbm{x}_0) \frac{\partial^2 f}{\partial \mathbbm{x}^{(j_1)}\partial \mathbbm{x}^{(j_2)}}\Bigg \vert_{\substack{\mathbbm{x}=\mathbbm{x}_0}}.
\end{equation*}
\end{theorem}

\begin{proof}
    Fix $\lambda\in T^n(E)$, we consider 
    \begin{align*}
    \mathbb{L}_n : \mathbb{R}\times \mathbb{R}^{D_n} \to \mathbb{C},\quad (t,\mathbbm{x}; \lambda) \mapsto \mathbb{E}[\exp(i M_{\lambda}(\mathbb{X}^n_t - \mathbb{X}^n_0)\rangle) | \mathbb{X}^n_0 = \mathbbm{x}]. 
\end{align*}
First, we note that $\mathbb{L}_n$ does not depend on $\mathbbm{x}_n$ i.e. the last degree of the $n$-th truncated tensor algebra. Indeed, using Taylor expansion, we can express $\mathbb{L}_n(t,\mathbbm{x}; \lambda)$ as a power series in $\Phi_m(t, \mathbbm{x})$
\begin{equation*}
    \mathbb{L}_n(t,\mathbbm{x}; \lambda) = \sum_{m>0} i^m M_{\lambda}^{\otimes m}(\Phi_m(t, \mathbbm{x})),
\end{equation*}
Provided the series in the RHS converges. From now on, without explicitly stating, we assume $\mathbbm{x}\in T^{n-1}(E)$.
By Theorem \ref{theorem enhanced pde}, for each $m>2$, $M_{\lambda}^{\otimes m}(\Phi_m(t, \mathbbm{x}))$ satisfies the recursive PDE stated in Equation \eqref{eq:recursive pde linear process} and $M_{\lambda}(\Phi_1(t, \mathbbm{x}))$ satisfies Equation \eqref{eq:recursive pde linear process 2}. Note that $M_{\lambda}^{\otimes 0}(\Phi_0(t, \mathbbm{x})) = 1$, 
summing up the recursive PDE equations for $M_{\lambda}^{\otimes m}(\Phi_m(t, \mathbbm{x}))$  over $m>2$ we get
\begin{eqnarray*}
\left(-\frac{\partial }{\partial t}+ A\right)\mathbb{L}_n(t,\mathbbm{x}; \lambda) +i\left((M_{\lambda}\circ \mu_n)(\mathbbm{x})+\sum_{j=1}^{D_{n-1}}(M_{\lambda}\circ b_n^{(\cdot, j)})(\mathbbm{x})\frac{\partial }{\partial \mathbbm{x}^{(j)}}\right)\mathbb{L}_n(t,\mathbbm{x}; \lambda)\notag
\\-\frac{1}{2}(M_{\lambda}^{\otimes 2}\circ b_n)(\mathbbm{x})\mathbb{L}_n(t,\mathbbm{x}; \lambda)=0,
\end{eqnarray*}
for all $(t, \mathbbm{x})\in \R\times T^{n-1}(E)$ with $\mathbb{L}_n(0, \mathbbm{x};\lambda) =1$. 
\end{proof}

\begin{remark}
Solving the PDE stated in Equation \eqref{eq:general pde radius of convergence} analytically is often unfeasible and requires numerical schemes. However, using the Taylor series, we can extract information about $\mathbb{X}^n_t$ from the expected signature by taking the derivatives from the expansion.
\end{remark}

\begin{acks}
TL and HN are supported by the EPSRC [grant number
\\
EP/S026347/1], in
part by The Alan Turing Institute under the EPSRC grant 
\\
EP/N510129/1. TL is also funded in part by the Data Centric Engineering
Programme (under the Lloyd’s Register Foundation grant G0095), the Defence and Security Programme
(funded by the UK Government) and the Office for National Statistics \& The Alan Turing Institute
(strategic partnership) and in part by the Hong Kong Innovation and Technology Commission (InnoHK
Project CIMDA). The authors are grateful for the useful discussion with Dr. Sam Morley, Dr. Siran Li, and William Turner. 
\end{acks}


\end{document}